\documentclass[10pt]{amsart}
\usepackage{amsmath}
\usepackage{amsfonts}
\usepackage{amssymb}
\usepackage{tikz}
\usepackage{wasysym}
\usepackage{stmaryrd}
\usepackage{ytableau}
\usepackage{float}
\usepackage{latexsym, amscd, amsbsy, xypic, color,verbatim}
\usepackage{mathrsfs,bbm,dsfont}
\usepackage{graphicx, graphics, pstricks}
\usepackage{caption}
\usepackage[mathscr]{eucal}
\usepackage[all]{xy}
\usepackage{fancyhdr}
\usepackage{setspace}
\usepackage{subfigure}
\def\thefootnote{\fnsymbol{footnote}}
\newtheorem{thm}{Theorem}[section]
{}
\newtheorem{prop}[thm]{Proposition}
\newtheorem{lemma}[thm]{Lemma}

\newtheorem{defin}[thm]{Definition}
\newenvironment{definition}{\begin{defin} \rm}{\end{defin}}

\newcommand{\hookdownarrow}{\mathrel{\rotatebox[origin=c]{-90}{$\hookrightarrow$}}}

\newcommand{\arcd}{\ar@{-}@/_/} % Gives a down-curved arc, give direction as [dr] for down right
\newcommand{\arcu}{\ar@{-}@/^/} % Gives an up- curved arc, give direction as [dr] for down right
\newcommand{\arcU}{\ar@{-}@/^10pt/} % Gives an up- curved arc, give direction as [dr] for down right
\newcommand{\tra}{\ar@{-}} % Gives a straight through arc, give direction as [dr] for down right

\newcommand{\Hom}{\text{\textnormal{Hom}}}

  \def\De{\Delta}

 \def\de{\delta}

  \def\leq{\leqslant}  \def\geq{\geqslant}

\def\Hom{\mbox{\rm Hom}}  
 
 \def\res{\mbox{\rm Res}}
  \def\ind{\mbox{\rm Ind}}

\def\dim{\mbox{\rm dim}\,}

\def\ggp#1#2{\left[\kern-3.2pt\left[{#1\atop #2}\right]\kern-3.2pt\right]}

\def\la{\lambda}

\def\de{\delta}

\def\pr{\prime}

\def\o{\otimes}

\usetikzlibrary{positioning}
\tikzset{>=stealth}

\begin{document}

\title[Grothendieck ring of walled Brauer algebras]
{On the Grothendieck ring of the sequence of walled Brauer algebras}

\thanks {}

\author{Pei Wang}

\address{Wang: Teachers’ College, Beijing Union University, Beijing 100101, P. R. China}

\email{wangpei19@163.com}

\author{Yanbo Li$^{\dag}$}

\address{Li: School of Mathematics and Statistics, Northeastern
University at Qinhuangdao, Qinhuangdao, 066004, P.R. China}

\email{liyanbo707@163.com}

\maketitle

\begin{abstract}
In this paper, we provide a diagrammatic approach to study the branching rules for cell modules
on a sequence of walled Brauer  algebras. This approach also allows us to calculate the structure constants
of multiplication over the Grothendieck ring of the sequence.
\end{abstract}

\renewcommand{\thefootnote}{\alph{footnote}}
\setcounter{footnote}{-1} \footnote{$^{\dag}$Corresponding author: liyanbo707@163.com}
\renewcommand{\thefootnote}{\alph{footnote}}
\setcounter{footnote}{-1} 
\renewcommand{\thefootnote}{\alph{footnote}}
\setcounter{footnote}{-1} 

\section{Introduction}

The walled Brauer algebra
$B_{r,s}(\de)$ is a subalgebra of the Brauer algebra $B_{r+s}(\de)$,
which was introduced
independently by Koike \cite{Koi} and Turaev \cite{Tur}.
As  a generalization of the classical Schur-Weyl duality, the centralizer
of the natural action of $\mathrm{GL}_n(\mathbb{C})$ on a mixed tensor space
$V^{\o r}\o W^{\o s}$, with $V=\mathbb{C}^n$ and $W=V^*$, was characterized as
the walled Brauer algebra $B_{r,s}(n)$, for $n \geq r + s$.
Walled Brauer algebras have been extensively studied, including the cellularity,
semi-simplicity, decomposition numbers, Jucys-Murphy elements,
block theory, Kazhdan-Lusztig theory and so on. We refer the reader to
\cite{And, BS, cox1, cox2, JK, SS}  for details.

An important feature of a diagram algebra is that
a low-dimensional algebra can
be injected naturally as a subalgebra of a high-dimensional algebra.
This enables one to consider diagram algebras in the tower framework.
Wang \cite{WP} studied a tower of Temperley-Lieb algebras $\mathsf{TL}_{n}$
and introduced the concept of walled modules  to study the branching rules
for cell modules of algebras $\mathsf{TL}_{m} ~\otimes
\mathsf{TL}_{n}$ and $\mathsf{TL}_{m+n}$.  Using this concept, the
structure constants of  multiplication over the Grothendieck group can be calculated.
The present study aimed to generalize Wang's result \cite{WP} to walled Brauer algebras.

Note that the tensor product of two Temperley-Lieb diagrams $A$ and $B$ is defined as the juxtaposition,
that is, diagram $A$ is to the left of diagram $B$.
However, this definition cannot be generalized to walled Brauer diagrams directly
since the juxtaposition is no longer a walled Brauer diagram. To overcome this obstacle,
we introduce the so-called twisted tensor product of two walled Brauer diagrams.
Interestingly, the twisted tensor product of walled Brauer algebras is actually isomorphic
to their ordinary tensor product. As a result, we obtain
the sequence of walled Brauer algebras  with  embedding
$$\rho_{r,s,n,m}: B_{r,s}(\delta) \o_k B_{n,m}(\delta)\hookrightarrow B_{r+n,s+m}(\delta).$$
Using the tensor product, we can investigate the restriction and induction functors
between the module categories of algebras $B_{r,s}(\delta) \o_k B_{n,m}(\delta) $ and
$B_{r+n,s+m}(\delta)$. In this process, the double-walled module,
which can be viewed as a generalization of walled modules, is of primary importance.
With these preparations, the structure constants of
the Grothendieck ring for a sequence of walled Brauer algebras is calculated finally.

The paper is organized as follows. In Section 2, we provide a quick review
of the cellular structure of a walled Brauer algebra.
In Section 3, we introduce the so-called twisted tensor product of
the walled Brauer algebra, which is proved to be isomorphic to the tensor product algebra.
In order to characterize the restriction of cell modules, we study in Section 4 the double walled modules,
and based on the main result of this section, we calculate the structure constants of
the Grothendieck ring for a sequence of walled Brauer algebras in Section 5.

\section*{Acknowledgment}
Part of this work was done when Wang visited Northeastern University at Qinhuangdao in 2019.
He takes this opportunity to express his sincere thanks to the School of Mathematics and Statistics for the hospitality during his
visit.

Wang is supported by NSFC (No. 11901033) and the General Project of Science and Technology Plan of Beijing Municipal Education Commission (No. KM202011417012). Li is supported by the Natural Science Foundation of Hebei
Province, China (A2017501003) and NSFC (No.11871107).

\section[Definition and Structure]{Preliminaries}\label{sec: definition}
In this section, we shall provide a quick review of the definition
and cellular structure of walled Brauer algebras.

\subsection{Definition of walled Brauer algebra}
Let $k$ be a field and $\delta\in k$. For $n\in \mathbb{N}$,
recall that \emph{Brauer algebra} $B_n(\delta)$ has a basis that includes all partitions of
$\{1,\ldots,n,\bar{1},\ldots,\bar{n}\}$ with each part in a partition just being a pair.
Such a partition can be represented by a so-called \textit{$n$-diagram} as follows.
The diagram consists of two rows of $n$ dots, with the top row labeled by $1,...,n$ in order from left to right,
and similarly, with the bottom row labeled by $\bar{1},...,\bar{n}$.
Then the dots belonging to the same part are joined with a smooth curve, i.e., an edge.
An edge is called a \emph{propagating edge} if it connects two dots in different rows;
else, it is called an {\em arc}.

The multiplication $A\cdot B$ of two $n$-diagrams $A$ and $B$ is given by concatenation, i.e., by
stacking $A$ on the top of $B$, identifying the bottom dots of $A$ with
the top dots of $B$, and following the lines from the top to bottom or within one row.
Note that the result diagram may contain some closed curses. The number of these closed curses is denoted by $t$. Then
the multiplication $A\cdot B$ is
defined to be $\delta^tC$, where $C$ is the diagram obtained by removing the aforementioned closed curses.

For two natural numbers $r$ and $s$, the \emph{walled Brauer algebra}
$B_{r,s}(\delta)$ (sometimes denoted by $B_{r,s}$ briefly) is a subalgebra of $B_{r+s}(\delta)$.
Its a basis consists of the so-called walled Brauer diagrams
satisfying certain conditions. Given an $r+s$ Brauer diagram, a vertical line (wall) is added to separate
the first $r$ top and bottom dots from the remainder. Then, an $(r, s)$-walled Brauer diagram requires that
no propagating edge crosses the wall and that every arc does cross the wall.
Let us illustrate the product of two walled Brauer diagrams when $n = 8$, $r=3$, and $s=5$ (see Figure 1).

\begin{figure}[H]
	\begin{equation*}
	D_1\cdot D_2=
	\begin{array}{c}
	\begin{tikzpicture}
	\node at (1.3,1.8) {$D_1$};
	\node at (0,1.4) {$1$};
	\node at (0.8,1.4) {$3$};
	\node at (0.42,1.4) {$2$};
	\node at (1.6,1.4) {$4$};
	\node at (2.4,1.4) {$\cdots$};
	\node at (3.2,1.4) {$8$};
	\node at (0,0) {\tiny\textbullet};
	\node at (0,0.9) {\tiny\textbullet};
	\node at (0.42,0) {\tiny\textbullet};
	\node at (0.42,0.9) {\tiny\textbullet};
	\node at (0.8,0) {\tiny\textbullet};
	\node at (0.8,0.9) {\tiny\textbullet};
	\node at (1.6,0) {\tiny\textbullet};
	\node at (1.6,0.9) {\tiny\textbullet};
	\draw (1.2,1.3) --(1.2,-0.3);
	\node at (2.0,0) {\tiny\textbullet};	
	\node at (2.0,0.9) {\tiny\textbullet};
	\node at (2.4,0) {\tiny\textbullet};	
	\node at (2.4,0.9) {\tiny\textbullet};
	\node at (2.8,0) {\tiny\textbullet};	
	\node at (2.8,0.9) {\tiny\textbullet};
	\node at (3.2,0) {\tiny\textbullet};
	\node at (3.2,0.9) {\tiny\textbullet};
	\draw (0,0.9) --(0.42,0);
	\draw (0.4,0.9) .. controls (1,0.4) and (1.8,0.4) .. (2.4,0.9);
	\draw (0.8,0.9) .. controls (1.2,0.5) and (2.4,0.5) .. (2.8,0.9);
	\draw (1.6,0.9) --(3.2,0);
	\draw (2,0.9) --(2.8,0);
	\draw (3.2,0.9) --(1.6,0);
	\draw (0,0) .. controls (0.5,0.5) and (1.5,0.5) .. (2,0);
	\draw (0.8,0) .. controls (1.2,0.4) and (2,0.4) .. (2.4,0);
%	\draw (1.6,0) --(0.8,0.9);
%	\draw (2.0,0) .. controls (2.35,0.48) and (3.25,0.48) .. (3.6,0);
%	\draw (1.2,0) .. controls (1.6,0.5) and (2.8,0.5) .. (3.2,0);
%	\draw [dotted] (2.4,-.2) --(2.4,1.0);
%	\draw (2.8,0) --(4.0,0.9);
%	\draw (4.0,0) --(4.4,0.9);
%%%%%%
\node at (0,-0.9) {\tiny\textbullet};
\node at (0.42,-0.9) {\tiny\textbullet};
\node at (0.8,-0.9) {\tiny\textbullet};
\node at (1.6,-0.9) {\tiny\textbullet};
\draw (1.2,1.3) --(1.2,-1.3);	
\node at (2.0,-0.9) {\tiny\textbullet};
\node at (2.4,-0.9) {\tiny\textbullet};	
\node at (2.8,-0.9) {\tiny\textbullet};
\node at (3.2,-0.9) {\tiny\textbullet};
\draw (0.42,0) --(0,-0.9);
\draw (0,0) .. controls (0.6,-0.5) and (1.8,-0.5) .. (2.4,0);
\draw (0.8,0) .. controls (1.2,-0.35) and (1.6,-0.35) .. (2,0);
\draw (1.6,0) --(2,-0.9);
\draw (2.8,0) --(1.6,-0.9);
\draw (3.2,0) --(3.2,-0.9);
\draw (0.4,-0.9) .. controls (0.8,-0.35) and (2,-0.35) .. (2.4,-0.9);
\draw (0.8,-0.9) .. controls (1.2,-0.35) and (2.4,-0.35) .. (2.8,-0.9);
\node at (1.3,-1.6) {$D_2$};
	\end{tikzpicture}
	\end{array} =\, \de
	\begin{array}{c}
	\begin{tikzpicture}
	\node at (0,1.4) {$1$};
\node at (0.8,1.4) {$3$};
\node at (0.42,1.4) {$2$};
\node at (1.6,1.4) {$4$};
\node at (2.4,1.4) {$\cdots$};
\node at (3.2,1.4) {$8$};
\node at (0,0) {\tiny\textbullet};
\node at (0,0.9) {\tiny\textbullet};
\node at (0.42,0) {\tiny\textbullet};
\node at (0.42,0.9) {\tiny\textbullet};
\node at (0.8,0) {\tiny\textbullet};
\node at (0.8,0.9) {\tiny\textbullet};
\node at (1.6,0) {\tiny\textbullet};
\node at (1.6,0.9) {\tiny\textbullet};
\draw (1.2,1.3) --(1.2,-0.3);
\node at (2.0,0) {\tiny\textbullet};	
\node at (2.0,0.9) {\tiny\textbullet};
\node at (2.4,0) {\tiny\textbullet};	
\node at (2.4,0.9) {\tiny\textbullet};
\node at (2.8,0) {\tiny\textbullet};	
\node at (2.8,0.9) {\tiny\textbullet};
\node at (3.2,0) {\tiny\textbullet};
\node at (3.2,0.9) {\tiny\textbullet};
\draw (0,0) --(0,0.9);
\draw (1.6,0.9) --(3.2,0);
\draw (2,0.9) --(1.6,0);
\draw (3.2,0.9) --(2,0);
\draw (0.4,0.9) .. controls (1,0.4) and (1.8,0.4) .. (2.4,0.9);
\draw (0.8,0.9) .. controls (1.2,0.5) and (2.4,0.5) .. (2.8,0.9);
\draw (0.4,0) .. controls (0.8,0.4) and (2,0.4) .. (2.4,0);
\draw (0.8,0) .. controls (1.2,0.4) and (2.4,0.4) .. (2.8,0);
	\end{tikzpicture}
	\end{array}
	\end{equation*}
	\vskip -0.3cm
	\caption{Multiplication of two walled Brauer diagrams}
	\label{fig:vect}
\end{figure}
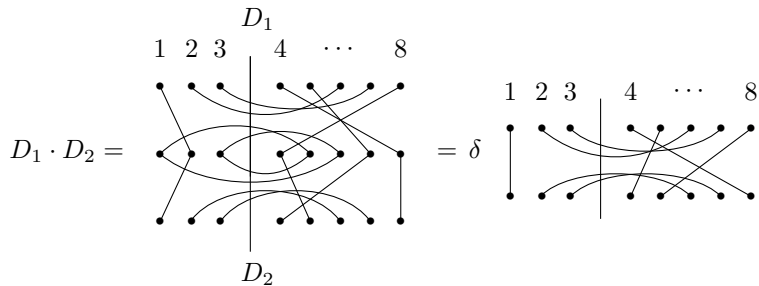

The following lemma on the multiplication of diagram algebras is well-known.
We write it here for the sake of description convenience later.
It is helpful to point out that the lemma is valid for not only walled Brauer diagrams, but also
Temperley-Lieb diagrams, partition diagrams, and so on.

\begin{lemma}\label{capnever}
The number of propagating edges
never increases and the number of  arcs never decreases in a concatenation of two diagrams.
\end{lemma}

\subsection{Cellular structures}

Cellular algebras were introduced by Graham and Lehrer \cite{GL} in 1996 in order to study the
non-semisimple specializations of many important algebras, including Hecke
algebras \cite{G}, Brauer algebras, Birman-Wenzl algebras \cite{X} and so on.
In \cite{cox2}, Cox et al. proved that the walled Brauer algebras are cellular.
In this section, we review the cellular structure.

Denote by $\Sigma_n$  the
symmetric group on $n$ letters. Evidently, each element in $\Sigma_n$
can be represented by an $n$-diagram without arcs.
Denote $\Sigma_i \times \Sigma_j$ by $\Sigma_{i,j}$.
If we regard the  diagrams in $\Sigma_{i,j}$ as the juxtaposition of
diagrams in $\Sigma_i$ and $  \Sigma_j$, then $\Sigma_{i,j}$ is a subset
of $\Sigma_{i+j}$.

Moreover, we need to define a $k$-linear space consisting of half diagrams. Given an $(r, s)$-walled diagram
with $l$ arcs in each row, we can obtain an upper ``half diagram" by cutting the diagram
horizontally in half.
Considering the example of $D_1$ (subsection 2.1), we get an upper $(3, 5, 2)$-half diagram as follows.
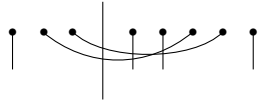
\begin{figure}[H]
	\begin{equation*}
	\begin{array}{c}
	\begin{tikzpicture}
	\node at (0,0.9) {\tiny\textbullet};
	\node at (0.42,0.9) {\tiny\textbullet};
	\node at (0.8,0.9) {\tiny\textbullet};
	\node at (1.6,0.9) {\tiny\textbullet};
	\draw (1.2,1.3) --(1.2,0);
	\node at (2.0,0.9) {\tiny\textbullet};
	\node at (2.4,0.9) {\tiny\textbullet};
	\node at (2.8,0.9) {\tiny\textbullet};
	\node at (3.2,0.9) {\tiny\textbullet};
	\draw (0,0.4) --(0,0.9);
	\draw (1.6,0.9) --(1.6,0.4);
	\draw (2,0.9) --(2,0.4);
	\draw (3.2,0.9) --(3.2,0.4);
	\draw (0.4,0.9) .. controls (1,0.4) and (1.8,0.4) .. (2.4,0.9);
	\draw (0.8,0.9) .. controls (1.2,0.5) and (2.4,0.5) .. (2.8,0.9);
	\end{tikzpicture}
	\end{array}
	\end{equation*}
 \caption{Half diagram}
\label{Fig}
\end{figure}
Denote by $V_{r,s}^l$ the $k$-linear space spanned by all $(r, s,l)$-half diagrams.
Then $V_{r,s}^l$ has a natural $B_{r,s}(\de)$-module structure with $B_{r,s}(\de)$-action
defined by the multiplication of diagrams. Hereinafter, $V_{r,s}^l$ is called an \textit{$(r, s,l)$-partial diagram module.}
Note that the result of this action is zero when
the number of propagating edges decreases.

\smallskip

In \cite{cox2}, the walled Brauer algebra $B_{r,s}(\delta)$ was
proved to be the \textit{iterated inflation} of group algebras $k\Sigma_{i,j}$
along $V_{r, s}^l$, and consequently, it is cellular
according to the inflation theory introduced by Koenig and Xi in \cite{KX2}.
More precisely, there is a chain of two-sided ideals
${0}= J_0 \subseteq J_1 \subseteq ... \subseteq J_n =B_{r,s}(\delta)$ such that
each subquotient $J_l/J_{l-1}$ (called a \textit{layer}) is a non-unital algebra
of the form $ V_{r,s}^l \otimes V_{r,s}^l \o k\Sigma_{r-l,s-l}$.
For arbitrary $x \in J_l/J_{l-1}$ and $y \in J_k/J_{k-1}$, we have $xy \in J_{\min\{l,k\}}$, that is,
the product cannot move to a higher layer.

Now, we  give more details about this cellular structure.
Let $\lambda^{L}=(\lambda_{l_{1}}^{L}, \lambda_{l_{2}}^{L}, \ldots, \lambda_{l_{k}}^{L})$
be a partition of $r-l$,
and $\lambda^{R}=(\lambda_{l_{1}}^{R}, \lambda_{l_{2}}^{R}, \ldots, \lambda_{l_{k^{'}}}^{R})$
a partition of $s-l$.
It is well-known that the cell modules
of $k\Sigma_{r-l,s-l}$ are precisely those modules of the form $S^{\la^L} \boxtimes S^{\la^R}$,
where $S^{\la^L}$ is a Specht module of $k\Sigma_{r-l}$ and $S^{\la^R}$ a Specht module of $k\Sigma_{s-l}$,
and thus the modules can be labeled by pairs
$(\lambda^{L},\lambda^{R})$.
For each integer $0\leqslant l\leqslant \min(r,s)$, we set
\begin{align*}
&\!\!\!\! \Lambda_{r,s}(l):=\{ (\lambda^L,\lambda^R) \mid \lambda^L\vdash r-l ; \lambda^R\vdash s-l\}\ \text{and}\ \Lambda_{r,s}:=\!\!\!\!\!\bigcup_{l=0}^{\min(r,s)}\!\!\Lambda_{r,s}(l).
\end{align*}
Then, the cell modules of algebra $B_{r,s}(\delta)$ are indexed by  set $\Lambda_{r,s}$ (see \cite{cox2} for details),
and cell module $\Delta_{r,s}(\la^L,\la^R)$
has the form $V_{r,s}^{l}\o (S^{\la^L}\boxtimes S^{\la^R})$.
The $B_{r,s}(\delta)$-action
on $ V_{r,s}^l\o (S^{\la^L}\boxtimes S^{\la^R})$ is defined as follows.

Given an $(r,s)$-walled diagram $x$ and a pure tensor $v\o s \in V_{r,s}^l\o (S^{\la^L}\boxtimes S^{\la^R})$, we define
$$x(v\o s)=\left\{\begin{array}{ll}
(xv)\o \pi(x,v)s& \text{if $xv$ has $l$ crossed arcs,}\\
0& \text{otherwise,}
\end{array}\right.$$
where $xv$ is as given above, and $\pi(x,v) \in \Sigma_{r-l,s-l}$ is the permutation on the labeled dots of $xv$.

\smallskip

To conclude this section, we generalize \cite[Proposition 3]{HP1}
from Brauer algebras to walled Brauer algebras.

\begin{prop}\label{aalc} Keep the notations as above.
	Let $M$ and $N$ be $k\Sigma_{r-l,s-l}$-modules, and let  $V_{r,s}^l$ be a partial diagram module of $B_{r,s}$. Then $$
	\Hom_{B_{r,s}} (V_{r,s}^l\o M, V_{r,s}^l\o N)\cong \Hom_{k\Sigma_{r-l,s-l}} (M, N).$$
	
	%	In particular, the $P_n$-module $M\o V_l$ is indecomposable if and only if $M$ is indecomposable.
\end{prop}
\begin{proof}
Denote by $v_0 \in V_{r,s}^l$ the half diagram
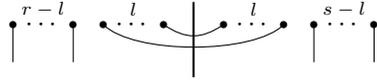
\begin{figure}[H]
	\begin{equation*}
	\begin{array}{c}
	\begin{tikzpicture}
	\node at (-0.4,0.9) {\tiny\textbullet};
	%	\draw (-0.4,0.9) .. controls (0.4,0.4) and (1.2,0.4) .. (2,0.9);
	\node at (-0.8,0.9) {$\cdots$};
	\node at (-0.8,1.1) {\scriptsize$r-l$};
	\draw (-0.4,0.4) --(-0.4,0.9);
	\node at (0,0.9) {\tiny\textbullet};
	\node at (-1.2,0.9) {\tiny\textbullet};
	\draw (-1.2,0.4) --(-1.2,0.9);
	\draw (0,0.9) .. controls (0.4,0.5) and (2,0.5) .. (2.4,0.9);
	%	\draw [thick,dashed] (0.2,0) --(0.2,1);
	
	%	\draw (0.4,0.9) .. controls (0.8,0.4) and (2.8,0.4) .. (3.2,0.9);
	\node at (0.8,0.9) {\tiny\textbullet};
	\node at (0.4,0.9) {$\cdots$};
	\node at (0.4,1.1) {\scriptsize$l$};
	\draw (0.8,0.9) .. controls (1.1,0.7) and (1.3,0.7) .. (1.6,0.9);
	\node at (1.6,0.9) {\tiny\textbullet};
	\draw[thick] (1.2,1.2) --(1.2,0.2);
	
	\node at (2.0,0.9) {$\cdots$};
	\node at (2,1.1) {\scriptsize$l$};
	
	\node at (2.4,0.9) {\tiny\textbullet};
	%	\draw [thick,dashed] (1,0) --(1,2);
	\node at (2.8,0.9) {\tiny\textbullet};
	
	\node at (3.2,0.9) {$\cdots$};
	\node at (3.2,1.1) {\scriptsize$s-l$};
	\node at (3.6,0.9) {\tiny\textbullet};

	\draw (2.8,0.9) --(2.8,0.4);
	\draw (3.6,0.9) --(3.6,0.4);
	
	\end{tikzpicture}
	\end{array}
	\end{equation*}
	\vskip -0.3cm
\end{figure}
and define an idempotent $e_{r,s,l}$ to be
\begin{equation*} e_{r,s,l}= \frac{1}{\de^{l}}
\begin{array}{c}
\begin{tikzpicture}
\node at (-0.4,0.9) {\tiny\textbullet};
\node at (-0.4,0) {\tiny\textbullet};
%	\draw (-0.4,0.9) .. controls (0.4,0.4) and (1.2,0.4) .. (2,0.9);
\node at (-0.8,0.9) {$\cdots$};
\node at (-0.8,0) {$\cdots$};
\node at (-0.8,1.1) {\scriptsize$r-l$};
\draw (-0.4,0) --(-0.4,0.9);
\node at (0,0.9) {\tiny\textbullet};
\node at (0,0) {\tiny\textbullet};
\node at (-1.2,0.9) {\tiny\textbullet};
\node at (-1.2,0) {\tiny\textbullet};
\draw (-1.2,0) --(-1.2,0.9);
\draw (0,0.9) .. controls (0.4,0.5) and (2,0.5) .. (2.4,0.9);
%	\draw [thick,dashed] (0.2,0) --(0.2,1);

%	\draw (0.4,0.9) .. controls (0.8,0.4) and (2.8,0.4) .. (3.2,0.9);
\node at (0.8,0.9) {\tiny\textbullet};
\node at (0.8,0) {\tiny\textbullet};
\node at (0.4,0.9) {$\cdots$};
\node at (0.4,0) {$\cdots$};
\node at (0.4,1.1) {\scriptsize$l$};
\draw (0.8,0.9) .. controls (1.1,0.7) and (1.3,0.7) .. (1.6,0.9);
\node at (1.6,0.9) {\tiny\textbullet};
\node at (1.6,0) {\tiny\textbullet};
\draw[thick] (1.2,1.2) --(1.2,-0.2);

\node at (2.0,0.9) {$\cdots$};
\node at (2.0,0) {$\cdots$};
\node at (2,1.1) {\scriptsize$l$};

\node at (2.4,0.9) {\tiny\textbullet};
\node at (2.4,0) {\tiny\textbullet};
%	\draw [thick,dashed] (1,0) --(1,2);
\node at (2.8,0.9) {\tiny\textbullet};
\node at (2.8,0) {\tiny\textbullet};
\node at (3.2,0.9) {$\cdots$};
\node at (3.2,0) {$\cdots$};
\node at (3.2,1.1) {\scriptsize$s-l$};
\node at (3.6,0.9) {\tiny\textbullet};	
\node at (3.6,0) {\tiny\textbullet};

\draw (0,0) .. controls (0.4,0.45) and (2,0.45) .. (2.4,0);
\draw (2.8,0.9) --(2.8,0);
\draw (3.6,0.9) --(3.6,0);
\draw (0.8,0) .. controls (1.1,0.2) and (1.3,0.2) .. (1.6,0);

\end{tikzpicture}
\end{array}
\end{equation*}

Let $\phi \in\Hom_{B_{r, s}} ( V_{r, s}^l\o M,  V_{r, s}^l\o N)$. For arbitrary $m\in M$, assume that
$$\phi( v_0 \o m)= \sum_{i\in I}   v_i\o n_i.$$ Apply the idempotent $e_{r,s,l}$
on both sides of the equality above. Note that $$e_{r,s,l}( v_0\o m)=   v_0\o m.$$ Moreover, each $e_{r,s,l}\cdot v_i$
is either zero or a multiple of $v_0$ and thus there exists certain $n\in N$ such that $$e_{r,s,l}.\sum v_i\otimes n_i= v_0\o n.$$
Consequently, we obtain a map $\hat{\phi}(m)=n$ from $M$ to $N$ induced by $\phi$ and it is easy to
check that $\hat{\phi}\in \Hom_{k\Sigma_{r-l,s-l}} (M, N)$.

On the other hand, for each $v \in  V_{r,s}^l$ there exists an element
$a \in B_{r, s}$ such that $v = a v_0 $.  Therefore, for each $m\in M$ and $v \in  V_{r,s}^l$ we have
\begin{eqnarray*}
	\phi( v \o m)	&=&  \phi((a v_0)\o m)=\phi((a v_0)\o \pi(a, v_0)\pi^{-1}(a, v_0)m)\\
	& = &a\phi( v_0\o \pi^{-1}(a, v_0)m)=a( v_0\o\hat{\phi}(\pi^{-1}(a, v_0)m))\\
	& = &v\o \pi(a, v_0)\hat{\phi}(\pi^{-1}(a, v_0)m)=v\o \hat{\phi}(m).
\end{eqnarray*}
That is, given a homomorphism $\hat{\phi}\in \Hom_{k\Sigma_{r-l,s-l}} (M, N)$,
we can define a map $\phi$ from previous equalities.
Furthermore, we claim that $\phi$ is a $B_{r,s}$-homomorphism. In fact, for arbitrary $d\in B_{r,s}$,
\begin{eqnarray*}
	d\phi(v\o m)	&=&  d(v\o\hat{\phi}(m))=dv\o (\pi(d,v)\hat{\phi}(m))\\
	& = & dv\o \hat{\phi}(\pi(d,v)m)=\phi(dv\o \pi(d,v)m)\\
	& = &\phi(d( v\o m))
\end{eqnarray*}
and this completes the proof.
\end{proof}

\section{ Tensor product and  twisted tensor product}

Let $B_{r,s}(\delta)$ and $B_{n,m}(\delta)$ be two walled Brauer algebras.
The aim of this section is to show that the tensor product algebra $B_{r,s}(\delta) \o_k B_{n,m}(\delta)$
is actually a subalgebra of $B_{r+n,s+m}(\delta)$.
The key is the so-called \textit{twisted tensor product} of walled Brauer algebras. We define this product in this section.
First, we need to define two kinds of mappings  $\iota$ and  $\zeta$,  which are used to extend a walled Brauer diagram to a bigger one.

Denote by $I_n$ the diagram of the identity of a Brauer algebra $B_n(\delta)$.
For $r,s,n,$ and $m\in \mathbb{N}$,  let $D$  be an $(r,s)$-walled  diagram in $B_{r,s}(\delta)$.
Define $\iota_{n,m}(D)$ to be an $(r+n,s+m)$-walled  diagram
in $B_{r+n,s+m}(\delta)$ obtained by inserting diagrams
$I_n$ and $I_m$ closing to the left and right sides of the wall, respectively, as shown in Fig.\ref{fig:3}(a).
Similarly, $\zeta_{n,m}(D)$ is an $(r+n,s+m)$-diagram, which is shown in Fig.\ref{fig:3}(b) obtained by juxtaposing diagrams $I_n$ and $I_m$ to  the left and  right sides of diagram $D$, respectively.

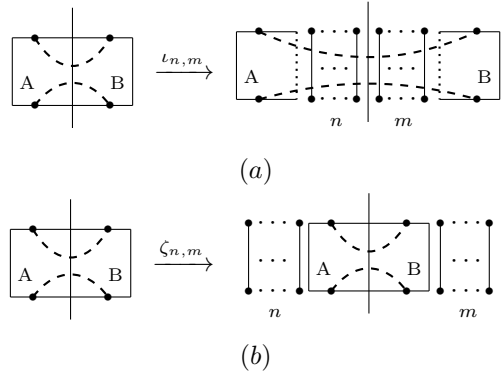
\begin{figure}[H]
	\begin{equation*}
	\begin{array}{c}
	\begin{tikzpicture}
\node at (0.3,0) {\tiny\textbullet};
\node at (0.3,0.9) {\tiny\textbullet};
\node at (1.3,0) {\tiny\textbullet};
\node at (1.3,0.9) {\tiny\textbullet};
\node at (0.2,0.3) {\scriptsize A};
\node at (1.4,0.3) {\scriptsize B};
\draw (0.8,1.3) --(0.8,-0.3);
\draw (0,0.9) --(1.6,0.9);
\draw (1.6,0) --(1.6,0.9);
\draw (0,0) --(0,0.9);
\draw (0,0) --(1.6,0);
\draw [thick,dashed] (0.3,0.9) .. controls (0.6,0.4) and (1,0.4) .. (1.3,0.9);
\draw [thick,dashed] (0.3,0) .. controls (0.6,0.4) and (1,0.4) .. (1.3,0);
	\end{tikzpicture}
	\end{array}\xrightarrow{\iota_{n,m}}
	\begin{array}{c}
	\begin{tikzpicture}
\node at (0.3,0) {\tiny\textbullet};
\node at (0.3,0.9) {\tiny\textbullet};
\node at (3.2,0) {\tiny\textbullet};
\node at (3.2,0.9) {\tiny\textbullet};
\node at (0.2,0.3) {\scriptsize A};
\node at (3.3,0.3) {\scriptsize B};
\draw [thick,dotted] (0.8,0.9) --(0.8,0);
\draw (0,0.9) --(0.8,0.9);
\draw (1.75,-0.3) --(1.75,1.3);
\draw [thick,dotted](2.7,0.9) --(2.7,0);
\draw (3.5,0.9) --(3.5,0);
\draw (0,0) --(0,0.9);
\draw (0,0) --(0.8,0);
\draw (2.7,0) --(3.5,0);
\draw (2.7,0.9) --(3.5,0.9);
\draw [thick,dashed] (0.3,0.9) .. controls (1.3,0.45) and (2.2,0.45) .. (3.2,0.9);
\draw [thick,dashed] (0.3,0) .. controls (1.3,0.28) and (2.2,0.28) .. (3.2,0);
\node at (1,0) {\tiny\textbullet};
\node at (1,0.9) {\tiny\textbullet};
\draw (1,0) --(1,0.9);
\node at (1.6,0) {\tiny\textbullet};
\node at (1.6,0.9) {\tiny\textbullet};
\draw (1.6,0) --(1.6,0.9);
\node at (1.9,0) {\tiny\textbullet};
\node at (1.9,0.9) {\tiny\textbullet};
\draw (1.9,0) --(1.9,0.9);
\node at (2.5,0) {\tiny\textbullet};
\node at (2.5,0.9) {\tiny\textbullet};
\draw (2.5,0) --(2.5,0.9);
\node at (1.33,0) { $\cdots$};
\node at (1.33,0.9) { $\cdots$};
\node at (1.33,0.4) { $\cdots$};
\node at (1.33,-0.3) {\scriptsize $n$};
\node at (2.22,0.4) { $\cdots$};
\node at (2.22,0) { $\cdots$};
\node at (2.22,0.9) { $\cdots$};
\node at (2.22,-0.3) {\scriptsize $m$};
	\end{tikzpicture}
	\end{array}	
	\end{equation*}
	$$(a)$$
	\begin{equation*}
\begin{array}{c}
\begin{tikzpicture}
\node at (0.3,0) {\tiny\textbullet};
\node at (0.3,0.9) {\tiny\textbullet};
\node at (1.3,0) {\tiny\textbullet};
\node at (1.3,0.9) {\tiny\textbullet};
\node at (0.2,0.3) {\scriptsize A};
\node at (1.4,0.3) {\scriptsize B};
\draw (0.8,1.3) --(0.8,-0.3);
\draw (0,0.9) --(1.6,0.9);
\draw (1.6,0) --(1.6,0.9);
\draw (0,0) --(0,0.9);
\draw (0,0) --(1.6,0);
\draw [thick,dashed] (0.3,0.9) .. controls (0.6,0.4) and (1,0.4) .. (1.3,0.9);
\draw [thick,dashed] (0.3,0) .. controls (0.6,0.4) and (1,0.4) .. (1.3,0);
\end{tikzpicture}
\end{array}\xrightarrow{\zeta_{n,m}}
\begin{array}{c}
\begin{tikzpicture}
\node at (0.3,0) {\tiny\textbullet};
\node at (0.3,0.9) {\tiny\textbullet};
\node at (1.3,0) {\tiny\textbullet};
\node at (1.3,0.9) {\tiny\textbullet};
\node at (0.2,0.3) {\scriptsize A};
\node at (1.4,0.3) {\scriptsize B};
\draw (0.8,1.3) --(0.8,-0.3);
\draw (0,0.9) --(1.6,0.9);
\draw (1.6,0) --(1.6,0.9);
\draw (0,0) --(0,0.9);
\draw (0,0) --(1.6,0);
\draw [thick,dashed] (0.3,0.9) .. controls (0.6,0.4) and (1,0.4) .. (1.3,0.9);
\draw [thick,dashed] (0.3,0) .. controls (0.6,0.4) and (1,0.4) .. (1.3,0);
\node at (-0.8,0) {\tiny\textbullet};
\node at (-0.8,0.9) {\tiny\textbullet};
\draw (-0.8,0) --(-0.8,0.9);
\node at (-0.12,0) {\tiny\textbullet};
\node at (-0.12,0.9) {\tiny\textbullet};
\draw (-0.12,0) --(-0.12,0.9);
\node at (1.75,0) {\tiny\textbullet};
\node at (1.75,0.9) {\tiny\textbullet};
\draw (1.75,0) --(1.75,0.9);
\node at (2.4,0) {\tiny\textbullet};
\node at (2.4,0.9) {\tiny\textbullet};
\draw (2.4,0) --(2.4,0.9);
\node at (-0.45,0) { $\cdots$};
\node at (-0.45,0.9) { $\cdots$};
\node at (-0.45,0.4) { $\cdots$};
\node at (-0.45,-0.3) {\scriptsize $n$};
\node at (2.12,0.4) { $\cdots$};
\node at (2.12,0) { $\cdots$};
\node at (2.12,0.9) { $\cdots$};
\node at (2.12,-0.3) {\scriptsize $m$};
\end{tikzpicture}
\end{array}
\end{equation*}
	$$(b)$$
	\vskip -0.3cm
	\caption{Schematic diagrams of mappings $\iota_{n,m}$ and $\zeta_{n,m}$}
	\label{fig:3}
\end{figure}

According to the definition of $\iota_{n,m}$ and $\zeta_{n,m}$, the following two lemmas are obvious.

\begin{lemma} Let $D_1, D_2$  be  $(r,s)$-walled  diagrams. Then
\begin{enumerate}
\item[(1)]\, $\iota_{n,m}(D_1 \cdot D_2)=
	\iota_{n,m}(D_1 )\cdot\iota_{n,m}( D_2)$
\item[(2)]\, $\zeta_{n,m}(D_1 \cdot D_2)=
	\zeta_{n,m}(D_1 )\cdot\zeta_{n,m}( D_2)$
\end{enumerate}
\end{lemma}

\begin{lemma}\label{3.2}
Let $r,s,n,m\in \mathbb{N}$. Let $D$ be an $(r, s)$-diagram and let $D^{'}$ be an $(n, m)$-diagram. Then
$$\iota_{n,m}(D)\cdot\zeta_{r,s}( D^{'})=\zeta_{r,s}(D^{'} )\cdot\iota_{n,m}( D)$$
\end{lemma}

Note that Lemma \ref{3.2} can be clarified from the following figure.
\begin{figure}[H]
	\begin{equation*}
	\begin{array}{c}
\begin{tikzpicture}
\node at (0.3,0) {\tiny\textbullet};
\node at (0.3,0.9) {\tiny\textbullet};
\node at (3.2,0) {\tiny\textbullet};
\node at (3.2,0.9) {\tiny\textbullet};
\node at (0.2,0.3) {\scriptsize A};
\node at (3.3,0.3) {\scriptsize B};
\draw [thick,dotted] (0.8,0.9) --(0.8,0);
\draw (0,0.9) --(0.8,0.9);
\draw (1.75,-0.3) --(1.75,1.2);
\draw [thick,dotted](2.7,0.9) --(2.7,0);
\draw (3.5,0.9) --(3.5,0);
\draw (0,0) --(0,0.9);
\draw (0,0) --(0.8,0);
\draw (2.7,0) --(3.5,0);
\draw (2.7,0.9) --(3.5,0.9);
\draw [thick,dashed] (0.3,0.9) .. controls (1.3,0.45) and (2.2,0.45) .. (3.2,0.9);
\draw [thick,dashed] (0.3,0) .. controls (1.3,0.28) and (2.2,0.28) .. (3.2,0);
\node at (1,0) {\tiny\textbullet};
\node at (1,0.9) {\tiny\textbullet};
\draw (1,0) --(1,0.9);
\node at (1.6,0) {\tiny\textbullet};
\node at (1.6,0.9) {\tiny\textbullet};
\draw (1.6,0) --(1.6,0.9);
\node at (1.9,0) {\tiny\textbullet};
\node at (1.9,0.9) {\tiny\textbullet};
\draw (1.9,0) --(1.9,0.9);
\node at (2.5,0) {\tiny\textbullet};
\node at (2.5,0.9) {\tiny\textbullet};
\draw (2.5,0) --(2.5,0.9);
\node at (1.33,0) { $\cdots$};
\node at (1.33,0.9) { $\cdots$};
\node at (1.33,0.4) { $\cdots$};
\node at (1.33,1.15) {\scriptsize $n$};
\node at (2.22,0.4) { $\cdots$};
\node at (2.22,0) { $\cdots$};
\node at (2.22,0.9) { $\cdots$};
\node at (2.22,1.15) {\scriptsize $m$};
\end{tikzpicture}\vspace{-5mm}\\
\begin{tikzpicture}
\node at (0.3,0) {\tiny\textbullet};
\node at (0.3,0.9) {\tiny\textbullet};
\node at (1.3,0) {\tiny\textbullet};
\node at (1.3,0.9) {\tiny\textbullet};
\node at (0.2,0.3) {\scriptsize C};
\node at (1.4,0.3) {\scriptsize D};
\draw (0.8,1.2) --(0.8,-0.3);
\draw (0,0.9) --(1.6,0.9);
\draw (1.6,0) --(1.6,0.9);
\draw (0,0) --(0,0.9);
\draw (0,0) --(1.6,0);
\draw [thick,dashed] (0.3,0.9) .. controls (0.6,0.4) and (1,0.4) .. (1.3,0.9);
\draw [thick,dashed] (0.3,0) .. controls (0.6,0.4) and (1,0.4) .. (1.3,0);
\node at (-0.9,0) {\tiny\textbullet};
\node at (-0.9,0.9) {\tiny\textbullet};
\draw (-0.9,0) --(-0.9,0.9);
\node at (-0.12,0) {\tiny\textbullet};
\node at (-0.12,0.9) {\tiny\textbullet};
\draw (-0.12,0) --(-0.12,0.9);
\node at (1.77,0) {\tiny\textbullet};
\node at (1.77,0.9) {\tiny\textbullet};
\draw (1.77,0) --(1.77,0.9);
\node at (2.5,0) {\tiny\textbullet};
\node at (2.5,0.9) {\tiny\textbullet};
\draw (2.5,0) --(2.5,0.9);
\node at (-0.47,0) { $\cdots$};
\node at (-0.47,0.9) { $\cdots$};
\node at (-0.47,0.4) { $\cdots$};
\node at (-0.47,-0.3) {\scriptsize $r$};
\node at (2.14,0.4) { $\cdots$};
\node at (2.14,0) { $\cdots$};
\node at (2.14,0.9) { $\cdots$};
\node at (2.14,-0.3) {\scriptsize $s$};
\end{tikzpicture}
	\end{array}=
	\begin{array}{c}
		\begin{tikzpicture}
	\node at (0.3,0) {\tiny\textbullet};
	\node at (0.3,0.9) {\tiny\textbullet};
	\node at (1.3,0) {\tiny\textbullet};
	\node at (1.3,0.9) {\tiny\textbullet};
	\node at (0.2,0.3) {\scriptsize C};
	\node at (1.4,0.3) {\scriptsize D};
	\draw (0.8,1.3) --(0.8,-0.3);
	\draw (0,0.9) --(1.6,0.9);
	\draw (1.6,0) --(1.6,0.9);
	\draw (0,0) --(0,0.9);
	\draw (0,0) --(1.6,0);
	\draw [thick,dashed] (0.3,0.9) .. controls (0.6,0.4) and (1,0.4) .. (1.3,0.9);
	\draw [thick,dashed] (0.3,0) .. controls (0.6,0.4) and (1,0.4) .. (1.3,0);
	\node at (-0.9,0) {\tiny\textbullet};
	\node at (-0.9,0.9) {\tiny\textbullet};
	\draw (-0.9,0) --(-0.9,0.9);
	\node at (-0.14,0) {\tiny\textbullet};
	\node at (-0.14,0.9) {\tiny\textbullet};
	\draw (-0.14,0) --(-0.14,0.9);
	\node at (1.75,0) {\tiny\textbullet};
	\node at (1.75,0.9) {\tiny\textbullet};
	\draw (1.75,0) --(1.75,0.9);
	\node at (2.5,0) {\tiny\textbullet};
	\node at (2.5,0.9) {\tiny\textbullet};
	\draw (2.5,0) --(2.5,0.9);
	\node at (-0.47,0) { $\cdots$};
	\node at (-0.47,0.9) { $\cdots$};
	\node at (-0.47,0.4) { $\cdots$};
	\node at (-0.47,1.2) {\scriptsize $r$};
	\node at (2.15,0.4) { $\cdots$};
	\node at (2.15,0) { $\cdots$};
	\node at (2.15,0.9) { $\cdots$};
	\node at (2.15,1.2) {\scriptsize $s$};
	\end{tikzpicture}\vspace{-5mm}\\
	\begin{tikzpicture}
	\node at (0.3,0) {\tiny\textbullet};
	\node at (0.3,0.9) {\tiny\textbullet};
	\node at (3.2,0) {\tiny\textbullet};
	\node at (3.2,0.9) {\tiny\textbullet};
	\node at (0.2,0.3) {\scriptsize A};
	\node at (3.3,0.3) {\scriptsize B};
	\draw [thick,dotted] (0.8,0.9) --(0.8,0);
	\draw (0,0.9) --(0.8,0.9);
	\draw (1.75,-0.3) --(1.75,1.3);
	\draw [thick,dotted](2.7,0.9) --(2.7,0);
	\draw (3.5,0.9) --(3.5,0);
	\draw (0,0) --(0,0.9);
	\draw (0,0) --(0.8,0);
	\draw (2.7,0) --(3.5,0);
	\draw (2.7,0.9) --(3.5,0.9);
	\draw [thick,dashed] (0.3,0.9) .. controls (1.3,0.45) and (2.2,0.45) .. (3.2,0.9);
	\draw [thick,dashed] (0.3,0) .. controls (1.3,0.28) and (2.2,0.28) .. (3.2,0);
	\node at (1,0) {\tiny\textbullet};
	\node at (1,0.9) {\tiny\textbullet};
	\draw (1,0) --(1,0.9);
	\node at (1.6,0) {\tiny\textbullet};
	\node at (1.6,0.9) {\tiny\textbullet};
	\draw (1.6,0) --(1.6,0.9);
	\node at (1.9,0) {\tiny\textbullet};
	\node at (1.9,0.9) {\tiny\textbullet};
	\draw (1.9,0) --(1.9,0.9);
	\node at (2.5,0) {\tiny\textbullet};
	\node at (2.5,0.9) {\tiny\textbullet};
	\draw (2.5,0) --(2.5,0.9);
	\node at (1.33,0) { $\cdots$};
	\node at (1.33,0.9) { $\cdots$};
	\node at (1.33,0.4) { $\cdots$};
	\node at (1.33,-0.3) {\scriptsize $n$};
	\node at (2.22,0.4) { $\cdots$};
	\node at (2.22,0) { $\cdots$};
	\node at (2.22,0.9) { $\cdots$};
	\node at (2.22,-0.3) {\scriptsize $m$};
	\end{tikzpicture}
	\end{array}
	\end{equation*}
\caption{Illustration of Lemma 3.2.}
\end{figure}
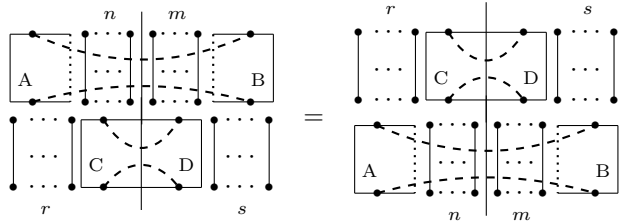

Now  we are in a position to give the definition of twisted tensor product of walled Brauer diagrams.
\begin{definition}  \label{tw}
Let $D$ be an $(r, s)$-diagram and let $D^{'}$ be an $(n, m)$-diagram.	
The \textit{twisted tensor product} $\boxtimes$ of diagrams $D$ and $D^{'}$ is an $(r+n,s+m)$-walled Brauer diagram defined by
	$$
	D \boxtimes D^{'}	:= \iota_{n,m}(D )\cdot\zeta_{r,s}( D^{'}).
	$$
\end{definition}

Figure 5 shows this definition more intuitively.

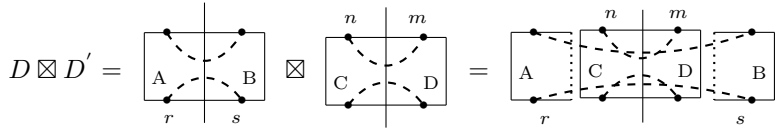
\begin{figure}[H]
\begin{equation*}D \boxtimes D^{'}=
\begin{array}{c}	\begin{tikzpicture}
\node at (0.3,0) {\tiny\textbullet};
\node at (0.3,0.9) {\tiny\textbullet};
\node at (1.3,0) {\tiny\textbullet};
\node at (1.3,0.9) {\tiny\textbullet};
\node at (0.2,0.3) {\scriptsize A};
\node at (1.4,0.3) {\scriptsize B};
\draw (0.8,1.3) --(0.8,-0.3);
\draw (0,0.9) --(1.6,0.9);
\draw (1.6,0) --(1.6,0.9);
\draw (0,0) --(0,0.9);
\draw (0,0) --(1.6,0);
\draw [thick,dashed] (0.3,0.9) .. controls (0.6,0.4) and (1,0.4) .. (1.3,0.9);
\draw [thick,dashed] (0.3,0) .. controls (0.6,0.4) and (1,0.4) .. (1.3,0);
\node at (0.33,-0.25) {\scriptsize $r$};

\node at (1.22,-0.25) {\scriptsize $s$};
\end{tikzpicture}
\end{array}\boxtimes
\begin{array}{c}
	\begin{tikzpicture}
\node at (0.3,0) {\tiny\textbullet};
\node at (0.3,0.9) {\tiny\textbullet};
\node at (1.3,0) {\tiny\textbullet};
\node at (1.3,0.9) {\tiny\textbullet};
\node at (0.2,0.3) {\scriptsize C};
\node at (1.4,0.3) {\scriptsize D};
\draw (0.8,1.3) --(0.8,-0.3);
\draw (0,0.9) --(1.6,0.9);
\draw (1.6,0) --(1.6,0.9);
\draw (0,0) --(0,0.9);
\draw (0,0) --(1.6,0);
\draw [thick,dashed] (0.3,0.9) .. controls (0.6,0.4) and (1,0.4) .. (1.3,0.9);
\draw [thick,dashed] (0.3,0) .. controls (0.6,0.4) and (1,0.4) .. (1.3,0);
\node at (0.33,1.15) {\scriptsize $n$};

\node at (1.22,1.15) {\scriptsize $m$};
\end{tikzpicture}
\end{array}=
\begin{array}{c}
\begin{tikzpicture}
\node at (0.3,0) {\tiny\textbullet};
\node at (0.3,0.9) {\tiny\textbullet};
\node at (3.2,0) {\tiny\textbullet};
\node at (3.2,0.9) {\tiny\textbullet};
\node at (0.2,0.36) {\scriptsize A};
\node at (3.3,0.36) {\scriptsize B};
\draw [thick,dotted] (0.8,0.9) --(0.8,0);
\draw (0,0.9) --(0.8,0.9);
\draw (1.75,-0.3) --(1.75,1.2);
\draw [thick,dotted](2.7,0.9) --(2.7,0);
\draw (3.5,0.9) --(3.5,0);
\draw (0,0) --(0,0.9);
\draw (0,0) --(0.8,0);
\draw (2.7,0) --(3.5,0);
\draw (2.7,0.9) --(3.5,0.9);
\draw [thick,dashed] (0.3,0.9) .. controls (1.3,0.54) and (2.2,0.54) .. (3.2,0.9);
\draw [thick,dashed] (0.3,0) .. controls (1.3,0.28) and (2.2,0.28) .. (3.2,0);

\node at (1.33,1.15) {\scriptsize $n$};

\node at (2.22,1.15) {\scriptsize $m$};
\end{tikzpicture}\vspace{-15.3mm}\\
\begin{tikzpicture}
\node at (0.3,0) {\tiny\textbullet};
\node at (0.3,0.9) {\tiny\textbullet};
\node at (1.3,0) {\tiny\textbullet};
\node at (1.3,0.9) {\tiny\textbullet};
\node at (0.2,0.36) {\scriptsize C};
\node at (1.4,0.36) {\scriptsize D};

\draw (0,0.9) --(1.6,0.9);
\draw (1.6,0) --(1.6,0.9);
\draw (0,0) --(0,0.9);
\draw (0,0) --(1.6,0);
\draw [thick,dashed] (0.3,0.9) .. controls (0.6,0.4) and (1,0.4) .. (1.3,0.9);
\draw [thick,dashed] (0.3,0) .. controls (0.6,0.4) and (1,0.4) .. (1.3,0);

\node at (-0.47,-0.3) {\scriptsize $r$};

\node at (2.13,-0.3) {\scriptsize $s$};
\end{tikzpicture}
\end{array}
\end{equation*}
	\vskip -0.3cm
	\caption{Twisted tensor product of walled Brauer diagrams}
	\label{fig:twist}
\end{figure}

The twisted tensor product is clearly associative. This fact is expressed in the form of the following lemma.

\begin{lemma}
The twisted tensor product $ \boxtimes $ is associative.
\end{lemma}

The following lemma implies that the twisted tensor product `$\boxtimes$' and
multiplication `$\cdot $' of diagrams are compatible. This enables us to accomplish the task of this section.

\begin{lemma}\label{compa}
Let $D_1 , D_2$  be $(r, s)$-walled Brauer
diagrams and let $D_1^{'} , D_2^{'}$  be $(n, m)$-walled Brauer diagrams.  Then
$$(D_1\boxtimes D_2) \cdot (D_{1}^{'}\boxtimes D_{2}^{'})=(D_1 \cdot D_2 ) \boxtimes ( D_{1}^{'} \cdot D_{2}^{'})$$
\end{lemma}

\begin{proof}
According to Definition \ref{tw} and Lemma \ref{3.2}, we have
\begin{eqnarray}
(D_1\boxtimes D_{1}^{'}) \cdot (D_2\boxtimes D_{2}^{'})	&=& \big(\iota_{n,m}(D_1 )
\cdot\zeta_{r,s}( D_1^{'})\big)\cdot\big(\iota_{n,m}(D_2 )\cdot\zeta_{r,s}( D_2^{'})\big)\nonumber\\
& = & \big(\iota_{n,m}(D_1 )\cdot\iota_{n,m}(D_2 )\big)\cdot\big( \zeta_{r,s}( D_1^{'})\cdot\zeta_{r,s}( D_2^{'})\big)\nonumber\\
& = & \iota_{n,m}(D_1 \cdot D_2 )\cdot \zeta_{r,s}( D_1^{'}\cdot D_2^{'})\nonumber\\
& = & (D_1 \cdot D_2 ) \boxtimes (D_{1}^{'} \cdot D_{2}^{'})\nonumber.
\end{eqnarray} The proof is completed.
\end{proof}

By extending the twisted tensor product of diagrams linearly to algebras $B_{r,s}(\delta)$
and $B_{n,m}(\delta)$, we obtain the twisted tensor product algebra $B_{r,s}(\delta) \boxtimes B_{n,m}(\delta)$.
Denote by  $B_{r,s}(\delta) \o_k B_{n,m}(\delta)$ the ordinary tensor product algebra.
Then we can give the main result of this section, which is derived directly from Lemma \ref{compa}.

\begin{thm}	\label{iso}
The twisted tensor product $B_{r,s}(\delta) \,\boxtimes\, B_{n,m}(\delta)$ is
a subalgebra of $B_{r+n,s+m}(\delta)$, and it
is isomorphic to  $B_{r,s}(\delta) \o_k B_{n,m}(\delta)$ as $k$-algebras.
\end{thm}

\begin{proof}
Twisted tensor product algebra $B_{r,s}(\delta) \,\boxtimes\, B_{n,m}(\delta)$ is clear
a subalgebra of $B_{r+n,s+m}(\delta)$.  So we only need to prove that it is isomorphic to $B_{r,s}(\delta) \o_k B_{n,m}(\delta)$.

Define a bilinear map $$g: B_{r,s}(\delta) \times B_{n,m}(\delta)\rightarrow B_{r,s}(\delta) \boxtimes B_{n,m}(\delta)$$ by
\[(	D
,	D^{'}	)
\mapsto
D \boxtimes D^{'}.
	\]
It is easy to check that the map $g$ is surjective and injective.
Denote by $$i:B_{r,s}(\delta)
\times B_{n,m}(\delta)\rightarrow B_{r,s}(\delta) \o B_{n,m}(\delta)$$ the canonical map.
Then there exists a unique $k$-linear mapping
$$\bar{g}:B_{r,s}(\delta) \o B_{n,m}(\delta)\rightarrow B_{r,s}(\delta) \boxtimes B_{n,m}(\delta)$$
satisfying $\bar{g}i=g$, which is bijective too.
Moreover, it is seen from Lemma \ref{compa} that map $\bar{g}$ is a $k$-algebra homomorphism, and we have proven the lemma.
\end{proof}

Clearly, a walled Brauer algebra $B_{r, s}(\delta)$ can be embedded into
$B_{r+1,s}(\delta)$ or $B_{r, s+1}(\delta)$ naturally.
These embeddings make the walled Brauer algebras form something looking like a bicomplex.

\begin{equation*}
\begin{array}{ccccccccc} B_{0,0}(\delta) &\hookrightarrow &B_{0,1}(\delta)
&\hookrightarrow& B_{0,2}(\delta) &\hookrightarrow& B_{0,3}(\delta)& \hookrightarrow& \cdots \\
\hookdownarrow & &\hookdownarrow & &\hookdownarrow  & &\hookdownarrow  & &   \\
B_{1,0}(\delta)&\hookrightarrow &B_{1,1}(\delta) &\hookrightarrow& B_{1,2}(\delta)
&\hookrightarrow& B_{1,3}(\delta)& \hookrightarrow& \cdots \\ \hookdownarrow
& &\hookdownarrow  & &\hookdownarrow  & &\hookdownarrow  & &   \\
B_{2,0}(\delta)&\hookrightarrow &B_{2,1}(\delta) &\hookrightarrow& B_{2,2}(\delta)
&\hookrightarrow& B_{2,3}(\delta)& \hookrightarrow& \cdots \\ \hookdownarrow
& &\hookdownarrow  & &\hookdownarrow  & &\hookdownarrow  & &  \\
\vdots
& &\vdots  & &\vdots  & &\vdots  & &  \\
\end{array}
\end{equation*}
Note that $B_{0,n}(\delta)=B_{n,0}(\delta)=k\Sigma_n$, the group algebra of the
symmetric group $\Sigma_n$ on $n$ letters.

If we denote by $G_{0}(B_{i,j}(\delta))$ the Grothendieck group of $B_{i,j}(\delta)$,
then we can obtain a bigraded $\mathbb{Z}$-module
$$G_{0}(B(\delta)):=\bigoplus\limits_{i,j\geqslant 0}G_{0}(B_{i,j}(\delta)).$$

Moreover, thanks to Theorem \ref{iso}, we have embeddings
$$\rho_{r,s,n,m}: B_{r,s}(\delta) \o_k B_{n,m}(\delta)\hookrightarrow B_{r+n,s+m}(\delta),$$
which enable us to define the multiplication for $G_{0}(B(\delta))$ and make $G_{0}(B(\delta))$ a ring.
In fact, let $M$ be a left $B_{r,s}(\delta)$-module, and  $N$ be a left $B_{n,m}(\delta)$-module.
Recall that tensor product $M\otimes_{k} N$ is a left
$B_{r,s}(\delta) \o_k B_{n,m}(\delta)$-module with action
$(a \otimes b)\cdot (w\otimes u)=aw \otimes bu$ for $a\in B_{r,s}(\delta), b\in B_{n,m}(\delta), w\in M$ and $u\in N$.
Then  in $G_{0}(B(\delta))$, the multiplication of the isomorphism classes of a $B_{r,s}(\delta)$-module $M$
and a $B_{n,m}(\delta)$-module $N$ can be defined by the induction product:
$$[M]\cdot [N] = [\ind_{\scriptsize{B_{r,s}(\delta)
\otimes B_{n,m}(\delta)}}^{\scriptsize{B_{r+n,s+m}(\delta)}}( M\otimes N)],$$  where
$$\ind_{\scriptsize{ B_{r,s}(\delta) \otimes B_{n,m}(\delta)}}^{\scriptsize{ B_{r+n,s+m}(\delta)}} (M\otimes N)
= B_{r+n,s+m}(\delta) \bigotimes_{ B_{r,s}(\delta) \otimes B_{n,m}(\delta)}(M\otimes N)$$
$$=\frac{ B_{r+n,s+m}(\delta) \otimes_k (M\otimes N)}{<a\otimes [(b\otimes c)(w\otimes u)]-
[a\rho_{r,s,n,m}(b\otimes c)]\otimes w\otimes u>}
$$
for $a\in  B_{r+n,s+m}(\delta), ~ b\in  B_{r,s}(\delta),~ c\in  B_{n,m}(\delta), w\in M$ and $u\in N$.

\smallskip

To ensure the multiplication
is well-defined and admits associativity, we require that the walled
Brauer algebras are semi-simple (see \cite[Theorem 3.5]{BL}).
We refer the reader to \cite[Theorem 6.3]{cox2} for relative results on semi-simplicity.

The structure constants of the Grothendieck ring are calculated in Section 5.

\bigskip

\section{Double-walled modules}\label{subsec: Double walled modules}

In this section, we
characterize  the restriction of cell modules. To this end, we introduce
\textit{double-walled modules}, which will be useful for
 describing the branching rules for cell modules
over a sequence of the  walled Brauer algebras.
We also give a composition series of the restricted module in Theorem \ref{keyp}.

\subsection{Double-walled diagram and actions}

The notions givin in the section are inspired by the following
observations.
Let $r_1, r_2, s_1, s_2 \in \mathbb{N}$  with $r_1 + r_2=r$ and $s_1 + s_2=s$.
For a partial one-row $(r, s,l)$ diagram $v\in V_{r,s}^{l}$, we can add two walls
between dots $r_1$, $r_1+1 $, and dots  $s_1$, $s_1+1 $, respectively. Then the
diagram is divided into four blocks, says $A, B, C$, and $D$ with $r_1, r_2, s_2$, and $s_1$ dots,
respectively (see the left diagram of Fig. \ref{moving} for an example).
We next endow the diagram with a $4$-\textit{tuple} $$(T_{AC},
T_{BD},T_{AD},T_{BC}),$$ wherein $T_{XY}$ is  the number of the arcs jointing from block $X$ to $Y$.
It is clear that  $T_{AC}+T_{BD}+T_{AD}+T_{BC}=l$.

For a diagram $v$, define a transformation $f$ by moving
its blocks $B, C$, together with the jointing arcs horizontally to the right side
of the block $D$. We illustrate $f$ as follows.

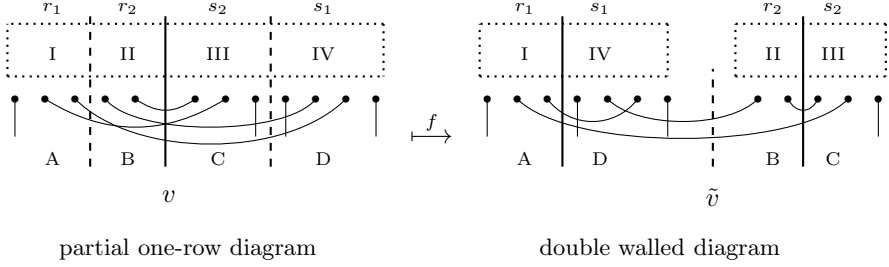
\begin{figure}[t]
	\begin{equation*}
	\begin{array}{c}
	\begin{tikzpicture}
%	\node at (0,1.4) {$1$};
%	\node at (0.8,1.4) {$3$};
%	\node at (0.42,1.4) {$2$};
%	\node at (1.6,1.4) {$4$};
%	\node at (2.4,1.4) {$\cdots$};
%	\node at (3.2,1.4) {$8$};
	\node at (-0.4,0.9) {\tiny\textbullet};
		\draw (-0.4,0.9) .. controls (0.4,0.4) and (1.2,0.4) .. (2,0.9);
	\node at (-0.8,0.9) {\tiny\textbullet};
		\draw (-0.8,0.4) --(-0.8,0.9);
	\node at (0,0.9) {\tiny\textbullet};
	\draw (0,0.9) .. controls (0.8,0.1) and (2.8,0.1) .. (3.6,0.9);
		\draw [thick,dashed] (0.2,0) --(0.2,2);
	\node at (0.4,0.9) {\tiny\textbullet};
	\draw (0.4,0.9) .. controls (0.8,0.4) and (2.8,0.4) .. (3.2,0.9);
	\node at (0.8,0.9) {\tiny\textbullet};
		\draw (0.8,0.9) .. controls (1.1,0.7) and (1.3,0.7) .. (1.6,0.9);
	\node at (1.6,0.9) {\tiny\textbullet};
	\draw[thick] (1.2,2) --(1.2,0);
	
	\node at (2.0,0.9) {\tiny\textbullet};
	
	\node at (2.4,0.9) {\tiny\textbullet};
		\draw [thick,dashed] (2.6,0) --(2.6,2);
	\node at (2.8,0.9) {\tiny\textbullet};
	
	\node at (3.2,0.9) {\tiny\textbullet};
	\node at (3.6,0.9) {\tiny\textbullet};	
	\node at (4,0.9) {\tiny\textbullet};

	\draw (2.4,0.9) --(2.4,0.4);
	\draw (4,0.9) --(4,0.4);
		\draw (2.8,0.9) --(2.8,0.4);
\draw [thick,dotted] (-0.9,1.2) --(-0.9,1.9);
\draw [thick,dotted] (-0.9,1.2) --(4.1,1.2);
\draw [thick,dotted] (4.1,1.9) --(4.1,1.2);
\draw [thick,dotted] (4.1,1.9) --(-0.9,1.9);
	\node at (0.7,1.5) {\scriptsize II};
\node at (-0.3,1.5) {\scriptsize I};
\node at (1.9,1.5) {\scriptsize III};
\node at (3.3,1.5) {\scriptsize IV};

	\node at (0.7,0.1) {\scriptsize B};
		\node at (0.7,2.1) {\scriptsize $r_2$};
	\node at (-0.3,0.1) {\scriptsize A};
		\node at (-0.3,2.1) {\scriptsize $r_1$};
		\node at (1.9,0.1) {\scriptsize C};
			\node at (1.9,2.1) {\scriptsize $s_2$};
	\node at (3.3,0.1) {\scriptsize D};
		\node at (3.3,2.1) {\scriptsize $s_1$};
		\node at (1.25,-0.4) { $v$};
		\node at (1.5,-1.1) { \small \mbox{partial one-row diagram}};
	\end{tikzpicture}
	\end{array}
 \stackrel{f}{\longmapsto}
	\begin{array}{c}
	\begin{tikzpicture}
%	\node at (0,1.4) {$1$};
%	\node at (0.8,1.4) {$3$};
%	\node at (0.42,1.4) {$2$};
%	\node at (1.6,1.4) {$4$};
%	\node at (2.4,1.4) {$\cdots$};
%	\node at (3.2,1.4) {$8$};
\draw [thick,dotted] (-0.9,1.2) --(-0.9,1.9);
\draw [thick,dotted] (-0.9,1.2) --(1.6,1.2);
\draw [thick,dotted] (1.6,1.9) --(1.6,1.2);
\draw [thick,dotted] (1.6,1.9) --(-0.9,1.9);
\node at (-0.3,1.5) {\scriptsize I};
\node at (0.7,1.5) {\scriptsize IV};
\node at (-0.4,0.9) {\tiny\textbullet};
\draw (-0.4,0.9) .. controls (0.3,0.2) and (3.3,0.2) .. (4,0.9);
\node at (-0.8,0.9) {\tiny\textbullet};
\draw (-0.8,0.4) --(-0.8,0.9);
\node at (0,0.9) {\tiny\textbullet};
\draw (0,0.9) .. controls (0.3,0.5) and (0.9,0.5) .. (1.2,0.9);
\draw [thick] (0.2,0) --(0.2,2);
\node at (0.4,0.9) {\tiny\textbullet};
\draw (0.4,0.9) --(0.4,0.4);
%\draw (0.4,0.9) .. controls (0.8,0.4) and (2.8,0.4) .. (3.2,0.9);
\node at (0.8,0.9) {\tiny\textbullet};
\draw (0.8,0.9) .. controls (1.2,0.5) and (2.4,0.5) .. (2.8,0.9);
\node at (1.2,0.9) {\tiny\textbullet};

\node at (1.6,0.9) {\tiny\textbullet};
\draw (1.6,0.9) --(1.6,0.4);
\draw[thick] (3.4,2) --(3.4,0);
\draw [thick,dashed] (2.2,0) --(2.2,1.3);

\node at (2.8,0.9) {\tiny\textbullet};

\node at (3.2,0.9) {\tiny\textbullet};
\draw (3.2,0.9) .. controls (3.3,0.7) and (3.5,0.7) .. (3.6,0.9);
\node at (3.6,0.9) {\tiny\textbullet};	
\node at (4,0.9) {\tiny\textbullet};
\node at (4.4,0.9) {\tiny\textbullet};

\draw (4.4,0.9) --(4.4,0.4);
\draw [thick,dotted] (2.5,1.2) --(2.5,1.9);
\draw [thick,dotted] (2.5,1.2) --(4.5,1.2);
\draw [thick,dotted] (4.5,1.9) --(4.5,1.2);
\draw [thick,dotted] (4.5,1.9) --(2.5,1.9);
\node at (3,1.5) {\scriptsize II};
\node at (3.8,1.5) {\scriptsize III};

\node at (0.7,0.1) {\scriptsize D};
	\node at (0.7,2.1) {\scriptsize $s_1$};
\node at (-0.3,0.1) {\scriptsize A};
	\node at (-0.3,2.1) {\scriptsize $r_1$};
\node at (3,0.1) {\scriptsize B};
	\node at (3,2.1) {\scriptsize $r_2$};
\node at (3.8,0.1) {\scriptsize C};
	\node at (3.8,2.1) {\scriptsize $s_2$};
	\node at (2.2,-0.4) { $\tilde{v}$};
		\node at (1.5,-1.1) { \small \mbox{double walled diagram}};
\end{tikzpicture}
	\end{array}
	\end{equation*}
	%\vskip -0.3cm
	\caption{Horizontally moving the blocks of diagram}
	\label{moving}
\end{figure}

We call the image of $f$ an $(r_1|s_1,r_2|s_2)$-\textit{double-walled diagram},
and denote it by symbol $\tilde{v}$. In fact, such a diagram
can also be obtained by first juxtaposing two partial one-row diagrams and
then possibly jointing single dots from block $A$ to $C$, or from $D$ to $B$.

We can define $B_{r_1,s_1}(\de)\o B_{r_2,s_2}(\de)$-action on an $(r_1|s_1,r_2|s_2)$-double-walled diagram as in
the case of a partial diagram module. Let us describe the action by using the right part of Fig. \ref{moving}.
Taking an $(r_1,s_1)$-diagram $D_1$ (with parts I and IV) and
an $(r_2,s_2)$-diagram $D_2$ (with parts II and III), then
 $D_1 \o D_2 $ can be viewed as the juxtaposition of these two diagrams.  Let $\tilde{v}$ be an
 $(r_1|s_1,r_2|s_2)$-double walled diagram with the total number of arcs $l$. The action is given by
concatenating diagrams $D_1$ with parts $AD$ and $D_2$ with parts $BC$.
Then, by identifying the corresponding dots gives a new  $(r_1|s_1,r_2|s_2)$-double-walled diagram $v'$
possibly with $t$ closed curses. Then
$$
(D_1\otimes D_2).\tilde{v}=\begin{cases} 0,& \text{if $t>l$;}\\
\delta^tv',\, &\text{otherwise.}\end{cases}
$$
Therefore, if we denote by $\widetilde{V^{l}}_{r_1|s_1,r_2|s_2}$ the linear space spanned by
all $(r_1|s_1,r_2|s_2)$-double-walled diagrams with $l$ arcs,
then it is a $B_{r_1,s_1}(\de)\o B_{r_2,s_2}(\de)$-module with the action defined above.

\begin{lemma}\label{updown} Let $r_1, r_2, s_1$ and $s_2 \in \mathbb{N}$
with $r_1 + r_2=r$ and $s_1 + s_2=s$.
Let $\tilde{v}$ be an $(r_1|s_1,r_2|s_2)$-double-walled diagram
with the tuple $(T_{AC},T_{BD},T_{AD},T_{BC})$, and the
total number of arcs being $l$. For an $(r_1, s_1)$-diagram
$D_1$ and an $(r_2, s_2)$-diagram $D_2$,
suppose that $(D_1 \o D_2)\cdot \tilde{v}$ has $l^\pr$  arcs
with the tuple $(T_{AC}^\pr,T_{BD}^\pr,T_{AD}^\pr,T_{BC}^\pr)$. Then $l\leqslant l^\pr$.
If $l=l^\pr$, then
\begin{enumerate}	
\item[(1)] the numbers $T_{AD}, T_{BC}$
 	of arcs jointing blocks  never  increases;
 	
\item[(2)] the numbers $T_{AC}, T_{BD}$ of arcs jointing blocks never decreases.
\end{enumerate}
\end{lemma}
\begin{proof}
Formula $l\leqslant l^\pr$ is a direct corollary of Lemma \ref{capnever}.
For (1-2), we first change the structure of a double-walled diagram
from ``left-right" to ``up-down" being illustrated as in Figure \ref{bimodule}.
The process is as follows. Given a double walled diagram $\tilde{v}$,
rotate the right side of the vertical dashed line (saying blocks $BC$)
clockwise, until two walls are connected with each other.
 \begin{figure}[t]
 	\begin{equation*}
 	\begin{array}{c}
 	\begin{tikzpicture}
 	%	\node at (0,1.4) {$1$};
 	%	\node at (0.8,1.4) {$3$};
 	%	\node at (0.42,1.4) {$2$};
 	%	\node at (1.6,1.4) {$4$};
 	%	\node at (2.4,1.4) {$\cdots$};
 	%	\node at (3.2,1.4) {$8$};
 	\draw [thick,dotted] (-0.9,1.2) --(-0.9,1.9);
 	\draw [thick,dotted] (-0.9,1.2) --(1.6,1.2);
 	\draw [thick,dotted] (1.6,1.9) --(1.6,1.2);
 	\draw [thick,dotted] (1.6,1.9) --(-0.9,1.9);
 	\node at (-0.3,1.5) {\scriptsize I};
 	\node at (0.7,1.5) {\scriptsize IV};
 	\node at (-0.4,0.9) {\tiny\textbullet};
 	\draw (-0.4,0.9) .. controls (0.3,0.2) and (3.3,0.2) .. (4,0.9);
 	\node at (-0.8,0.9) {\tiny\textbullet};
 	\draw (-0.8,0.4) --(-0.8,0.9);
 	\node at (0,0.9) {\tiny\textbullet};
 	\draw (0,0.9) .. controls (0.3,0.5) and (0.9,0.5) .. (1.2,0.9);
 	\draw [thick] (0.2,0) --(0.2,2);
 	\node at (0.4,0.9) {\tiny\textbullet};
 	\draw (0.4,0.9) --(0.4,0.4);
 	%\draw (0.4,0.9) .. controls (0.8,0.4) and (2.8,0.4) .. (3.2,0.9);
 	\node at (0.8,0.9) {\tiny\textbullet};
 	\draw (0.8,0.9) .. controls (1.2,0.5) and (2.4,0.5) .. (2.8,0.9);
 	\node at (1.2,0.9) {\tiny\textbullet};
 	
 	\node at (1.6,0.9) {\tiny\textbullet};
 	\draw (1.6,0.9) --(1.6,0.4);
 	\draw[thick] (3.4,2) --(3.4,0);
 	\draw [thick,dashed] (2.2,0) --(2.2,1.3);
 	
 	\node at (2.8,0.9) {\tiny\textbullet};
 	
 	\node at (3.2,0.9) {\tiny\textbullet};
 	\draw (3.2,0.9) .. controls (3.3,0.7) and (3.5,0.7) .. (3.6,0.9);
 	\node at (3.6,0.9) {\tiny\textbullet};	
 	\node at (4,0.9) {\tiny\textbullet};
 	\node at (4.4,0.9) {\tiny\textbullet};

 	\draw (4.4,0.9) --(4.4,0.4);
 	\draw [thick,dotted] (2.5,1.2) --(2.5,1.9);
 	\draw [thick,dotted] (2.5,1.2) --(4.5,1.2);
 	\draw [thick,dotted] (4.5,1.9) --(4.5,1.2);
 	\draw [thick,dotted,->] (2.5,1.9) --(4.5,1.9);
 	\node at (3,1.5) {\scriptsize II};
 	\node at (3.8,1.5) {\scriptsize III};
 	
 	\node at (0.7,0.1) {\scriptsize D};
 	\node at (-0.3,0.1) {\scriptsize A};
 	\node at (3,0.1) {\scriptsize B};
 	\node at (3.8,0.1) {\scriptsize C};
 	\node at (2.2,-0.4) { $\tilde{v}$};
 	
 	\draw[->,thick] (3.4,-0.2) arc(0:-120:1.5);
 	\end{tikzpicture}
 	\end{array}
 	\leftrightsquigarrow
 	\begin{array}{c}
 	\begin{tikzpicture}
 	%	\node at (0,1.4) {$1$};
 	%	\node at (0.8,1.4) {$3$};
 	%	\node at (0.42,1.4) {$2$};
 	%	\node at (1.6,1.4) {$4$};
 	%	\node at (2.4,1.4) {$\cdots$};
 	%	\node at (3.2,1.4) {$8$};
 	\draw [thick,dotted] (-0.9,1.2) --(-0.9,1.9);
 	\draw [thick,dotted] (-0.9,1.2) --(1.6,1.2);
 	\draw [thick,dotted] (1.6,1.9) --(1.6,1.2);
 	\draw [thick,dotted] (1.6,1.9) --(-0.9,1.9);
 	\node at (-0.3,1.5) {\scriptsize I};
 	\node at (0.7,1.5) {\scriptsize IV};
 	\node at (-0.4,0.9) {\tiny\textbullet};
 %	\draw (-0.4,0.9) .. controls (0.3,0.2) and (3.3,0.2) .. (4,0.9);
 	\node at (-0.8,0.9) {\tiny\textbullet};
 	\draw (-0.8,0.4) --(-0.8,0.9);
 	\node at (0,0.9) {\tiny\textbullet};
 	\draw (0,0.9) .. controls (0.3,0.5) and (0.9,0.5) .. (1.2,0.9);
 	\draw [thick] (0.2,-2) --(0.2,2);%wall
 	    \node at (0,-0.9) {\tiny\textbullet};
 		\node at (-0.4,-0.9) {\tiny\textbullet};
 	   	\node at (-0.8,-0.9) {\tiny\textbullet};
 	   		\node at (0.4,-0.9) {\tiny\textbullet};
 	   		\node at (0.8,-0.9) {\tiny\textbullet};
 	   		\draw (-0.8,-0.4) --(-0.8,-0.9);
 	   	\draw (0,-0.9) .. controls (0.1,-0.7) and (0.3,-0.7) .. (0.4,-0.9);
 	   	\draw (0.8,0.9) --(0.8,-0.9);
 	   		\draw (-0.4,0.9) --(-0.4,-0.9);
 	   	
 	   		\draw [thick,dotted] (-0.8,-1.2) --(-0.8,-1.9);
 	   	\draw [thick,dotted] (-0.8,-1.2) --(0.9,-1.2);
 	   	\draw [thick,dotted] (0.9,-1.9) --(0.9,-1.2);
 	   	\draw [thick,dotted,->] (0.9,-1.9) --(-0.8,-1.9);
 	   	\node at (0.55,-1.6) {\scriptsize II};
 	   	\node at (-0.3,-1.6) {\scriptsize III};

 	\node at (0.4,0.9) {\tiny\textbullet};
 	\draw (0.4,0.9) --(0.4,0.4);
 	%\draw (0.4,0.9) .. controls (0.8,0.4) and (2.8,0.4) .. (3.2,0.9);
 	\node at (0.8,0.9) {\tiny\textbullet};
% 	\draw (0.8,0.9) .. controls (1.2,0.5) and (2.4,0.5) .. (2.8,0.9);
 	\node at (1.2,0.9) {\tiny\textbullet};
 	
 	\node at (1.6,0.9) {\tiny\textbullet};
 	\draw (1.6,0.9) --(1.6,0.4);
% 	\draw[thick] (3.4,2) --(3.4,0);
 %	\draw [thick,dashed] (2.2,0) --(2.2,1.3);
 	
% 	\node at (2.8,0.9) {\tiny\textbullet};
 	
 %	\node at (3.2,0.9) {\tiny\textbullet};
% 	\draw (3.2,0.9) .. controls (3.3,0.7) and (3.5,0.7) .. (3.6,0.9);
% 	\node at (3.6,0.9) {\tiny\textbullet};	
 %	\node at (4,0.9) {\tiny\textbullet};
 	%\node at (4.4,0.9) {\tiny\textbullet};

%	\draw (4.4,0.9) --(4.4,0.4);

 	\node at (0.7,0.1) {\scriptsize D};
 	\node at (-0.3,0.1) {\scriptsize A};
 		\node at (0.7,-0.2) {\scriptsize B};
 	\node at (-0.3,-0.2) {\scriptsize C};

 %	\node at (2.2,-0.4) { $\tilde{v}$};
 	\end{tikzpicture}
 	\end{array}
 	\end{equation*}
 	\vskip -0.3cm
 	\caption{}
 	\label{bimodule}
 \end{figure}
Then (1-2) follow from the module action being essentially a concatenation of  diagrams.
\end{proof}

\noindent{\bf Remarks} (1) When $l=l'$, Lemma \ref{updown} implies that
 $T_{AC}$ or $ T_{BD}$ will be certain to decrease by one once $T_{AD}$ or $T_{BC}$ increases by one, and vice versa.

\smallskip

(2)  $V_{r,s}^{l}$ also admits a $B_{r_1,s_1}(\de)\o B_{r_2,s_2}(\de)$-action
in terms of the mapping  $\bar{g}$ in Theorem \ref{iso}:
For a partial diagram $v\in V_{r,s}^{l}$,  an $(r_1,s_1)$-diagram $D_1$ and
 an $(r_2,s_2)$-diagram $D_2$, define
 $$ (D_1 \o D_2).v:= \bar{g}(D_1 \o D_2).v=(D_1 \boxtimes D_2).v$$

As a $B_{r_1,s_1}(\de)\o B_{r_2,s_2}(\de)$-module,
$V_{r,s}^{l}$ is isomorphic to  $\widetilde{V^{l}}_{r_1|s_1,r_2|s_2}$. In fact,
the map $f$ defined in the beginning of this section is clearly invertible, and
$$ f((D_1 \o D_2).v )=f((D_1 \boxtimes D_2).v)
=(D_1 \o D_2).\tilde{v}=(D_1 \o D_2).f(v)$$

\smallskip

\subsection{Filtration}
% Since each partial one-row diagram has a wall, we use the term ``double walled''.
First of all,
we introduce a schematic manner to express partial diagrams:
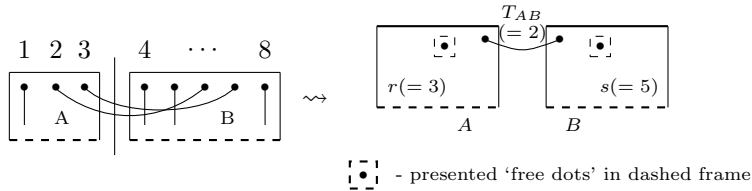
\begin{figure}[H]
	\begin{equation*}
	\begin{array}{c}
	\begin{tikzpicture}
	\node at (0,1.4) {$1$};
	\node at (0.8,1.4) {$3$};
	\node at (0.42,1.4) {$2$};
	\node at (1.6,1.4) {$4$};
	\node at (2.4,1.4) {$\cdots$};
	\node at (3.2,1.4) {$8$};
	\node at (0,0.9) {\tiny\textbullet};
	\node at (0.42,0.9) {\tiny\textbullet};
	\node at (0.8,0.9) {\tiny\textbullet};
	
	\node at (1.6,0.9) {\tiny\textbullet};
	\draw (1.2,1.3) --(1.2,0);
	
	\node at (2.0,0.9) {\tiny\textbullet};
	
	\node at (2.4,0.9) {\tiny\textbullet};
	
	\node at (2.8,0.9) {\tiny\textbullet};
	
	\node at (3.2,0.9) {\tiny\textbullet};
	\draw (0,0.4) --(0,0.9);
	\draw (1.6,0.9) --(1.6,0.4);
	\draw (2,0.9) --(2,0.4);
	\draw (3.2,0.9) --(3.2,0.4);
	\draw (0.4,0.9) .. controls (1,0.4) and (1.8,0.4) .. (2.4,0.9);
	\draw (0.8,0.9) .. controls (1.2,0.5) and (2.4,0.5) .. (2.8,0.9);
	\draw (1,1.1) --(-0.2,1.1);
	\draw (1,1.1) --(1,0.2);
	\draw (-0.2,1.1) --(-0.2,0.2);
	\draw [thick,dashed] (-0.2,0.2) --(1,0.2);
	\draw (1.4,1.1) --(1.4,0.2);
	\draw (1.4,1.1) --(3.4,1.1);
	\draw (3.4,1.1) --(3.4,0.2);
	\draw [thick,dashed](1.4,0.2) --(3.4,0.2);
	\node at (2.7,0.5) {\scriptsize B};
	\node at (0.5,0.5) {\scriptsize A};
	\end{tikzpicture}
	\end{array}
	\rightsquigarrow
	\begin{array}{c}
	\begin{tikzpicture}[scale=0.9]
	
	%%%%%%%%%%%%%%
	\node at (0.6,0.3) {\scriptsize$r(=3)$};
	\draw[thick] (0,1.2) --(1.8,1.2);
	\draw (0,0) --(0,1.2);
	\draw (1.8,0) --(1.8,1.2);
	\node at (1.3,-0.25) {\scriptsize$A$};
	%\node at (0.2,0.2) {$t$};
	\draw[thick,dashed]  (0,0) -- (1.8,0);
	%\node at (0.4,-.5) {$S^{\lambda_l^L}$};
	%%%%%%%%%%
	\node at (3.73,0.3) {\scriptsize$s (=5)$};
	\draw[thick] (2.5,1.2) --(4.3,1.2);
	\draw (2.5,0) --(2.5,1.2);
	\draw (4.3,0) --(4.3,1.2);
	\node at (2.9,-0.25) {\scriptsize$B$};
	%\node at (4,0.2) {$t$};
	\draw[thick,dashed]  (2.5,0) -- (4.3,0);
	%\node at (4,-.5) {$S^{\lambda_l^R}$};
	%%%%%%%%%%
	\node at (1.6,1) {\tiny\textbullet};
	\node at (1,0.9) {\tiny\textbullet};
	\draw[<->,dashed] (0.85,0.75) rectangle (1.15,1.05);
	\node at (2.7,1) {\tiny\textbullet};
	\node at (3.3,0.9) {\tiny\textbullet};
	\draw[<->,dashed] (3.15,0.75) rectangle (3.45,1.05);
	\draw (1.6,1) ..controls (2.15,0.8)..(2.7,1);
	\node at (2.15,1.4) {\scriptsize $T_{AB}$};
	\node at (2.15,1.1) {\scriptsize $(=2)$};
	\draw[thick,<->,dashed] (-0.4,-1.2) rectangle (0,-0.8);
	\node at (-0.2,-1) {\tiny\textbullet};
		\node at (2.9,-1) {\scriptsize- presented `free dots' in dashed frame};
	\end{tikzpicture}
	\end{array}
	\end{equation*}
	\vskip -0.3cm
	\caption{Schematic presentation}
	\label{fig:simple}
\end{figure}
Here, in the schematic presentation, `free dots' (that are not vertices of arcs)
are contained in a dashed frame with a single dot, and arcs crossing the wall become a single
one with number $T_{AB}$ ($=l$) on it. For example, changing Fig. \ref{moving} to the schematic presentation, as shown in Fig. \ref{schemoving},
the left diagram is acted by twisted product $B_{r_1,s_1}(\de)\boxtimes B_{r_2,s_2}(\de)$,
and the right one is acted by algebra $B_{r_1,s_1}(\de)\o B_{r_2,s_2}(\de)$.
 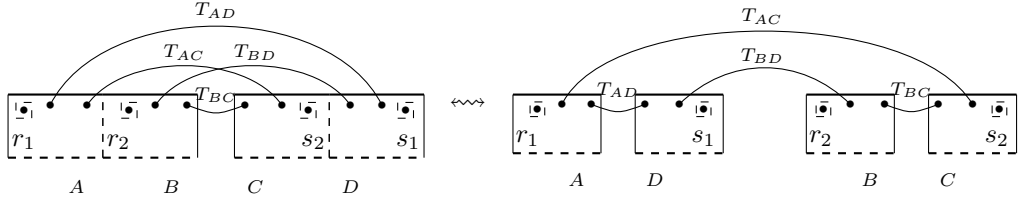
\begin{figure}[H]
 	\begin{equation*}
 	\begin{array}{c}
 	\begin{tikzpicture}[scale=0.7]
 	%%%%%%A
 	\node at (-1.5,0.3) {$r_1$};
 	\draw[thick] (-1.8,1.2) --(0,1.2);
 	\draw (-1.8,0) --(-1.8,1.2);
 	\draw[dashed] (0,0) --(0,1.2);
 	\node at (-0.5,-0.55) {\scriptsize$A$};
 	%\node at (-1.6,0.2) {$t$};
 	\draw[thick,dashed]  (-1.8,0)-- (0,0);
 	%\node at (-1.4,-.5) {$S^{\lambda^L}$};
 	
 	%%%%%%%%%%%%%%B
 	\node at (0.3,0.3) {$r_2$};
 	\draw[thick] (0,1.2) --(1.8,1.2);
 	\draw[dashed] (0,0) --(0,1.2);
 	\draw (1.8,0) --(1.8,1.2);
 	\node at (1.3,-0.55) {\scriptsize$B$};
 	%\node at (0.2,0.2) {$t$};
 	\draw[thick,dashed]  (0,0) -- (1.8,0);
 	%\node at (0.4,-.5) {$S^{\lambda^L}$};
 	%%%%%%%%%%C
 	\node at (4,0.3) {$s_2$};
 	\draw[thick] (2.5,1.2) --(4.3,1.2);
 	\draw (2.5,0) --(2.5,1.2);
 	\draw[dashed] (4.3,0) --(4.3,1.2);
 	\node at (2.9,-0.55) {\scriptsize$C$};
 	%\node at (4,0.2) {$t$};
 	\draw[thick,dashed]  (2.5,0) -- (4.3,0);
 	%\node at (4,-.5) {$S^{\lambda^R}$};
 	%%%%%%%%%%%%%%%%%D
 	\node at (5.8,0.3) {$s_1$};
 	\draw[thick] (4.3,1.2) --(6.1,1.2);
 	\draw[dashed] (4.3,0) --(4.3,1.2);
 	\draw (6.1,0) --(6.1,1.2);
 	\node at (4.7,-0.55) {\scriptsize$D$};
 	%\node at (5.8,0.2) {$t$};
 	\draw[thick,dashed]  (4.3,0) --(6.1,0);
 	%\node at (5.8,-.5) {$S^{\lambda^R}$};
 	%%%%%%%%%%BC
 	\node at (1.6,1) {\tiny\textbullet};
 	\node at (2.7,1) {\tiny\textbullet};
 	\draw (1.6,1) ..controls (2.15,0.8)..(2.7,1);
 	\node at (2.15,1.2) {\scriptsize$T_{BC}$};
 		\node at (0.5,0.9) {\tiny\textbullet};
 	\draw[<->,dashed] (0.35,0.75) rectangle (0.65,1.05);
 	%%%%%%%%%%AD
 	\node at (-1,1) {\tiny\textbullet};
 		\node at (-1.5,0.9) {\tiny\textbullet};
 	\draw[<->,dashed] (-1.65,0.75) rectangle (-1.35,1.05);
 	\node at (5.3,1) {\tiny\textbullet};
 	\draw (-1,1) ..controls (0,3) and (4.3,3)..(5.3,1);
 	\node at (2.15,2.8) {\scriptsize$T_{AD}$};
 	%%%%%%%%%%AC
 	\node at (-0.3,1) {\tiny\textbullet};
 	\node at (3.4,1) {\tiny\textbullet};
 		\node at (3.9,0.9) {\tiny\textbullet};
 	\draw[<->,dashed] (3.75,0.75) rectangle (4.05,1.05);
 	\draw (-0.3,1) ..controls (0.4,2) and (2.7,2)..(3.4,1);
 	\node at (1.55,2) {\scriptsize$T_{AC}$};
 	%%%%%%%%%%BD
 	\node at (1,1) {\tiny\textbullet};
 	\node at (4.7,1) {\tiny\textbullet};
 	\node at (5.75,0.9) {\tiny\textbullet};
 	\draw[<->,dashed] (5.6,0.75) rectangle (5.9,1.05);
 	\draw (1,1) ..controls (1.7,2) and (4,2)..(4.7,1);
 	\node at (2.9,2) {\scriptsize$T_{BD}$};
 	\end{tikzpicture}
 	\end{array}
 	\leftrightsquigarrow
 	\begin{array}{c}
 	\begin{tikzpicture}[scale=0.65]
 	
 	%%%%%%%%%%%%%%A
 	\node at (0.3,0.3) {$r_1$};
 	\draw[thick] (0,1.2) --(1.8,1.2);
 	\draw (0,0) --(0,1.2);
 	\draw (1.8,0) --(1.8,1.2);
 	\node at (1.3,-0.55) {\scriptsize$A$};
 	%\node at (0.2,0.2) {$t$};
 	\draw[thick,dashed]  (0,0) --(1.8,0);
 	%\node at (0.4,-.5) {$S^{\lambda^L}$};
 	%%%%%%%%%%D
 	\node at (3.9,0.3) {$s_1$};
 	\draw[thick] (2.5,1.2) --(4.3,1.2);
 	\draw (2.5,0) --(2.5,1.2);
 	\draw (4.3,0) --(4.3,1.2);
 	\node at (2.9,-0.55) {\scriptsize$D$};
 	%\node at (4,0.2) {$t$};
 	\draw[thick,dashed]  (2.5,0) --(4.3,0);
 	%\node at (4,-.5) {$S^{\lambda^R}$};

 	%%%%%%%%%%%%%%B
 	\node at (6.3,0.3) {$r_2$};
 	\draw[thick] (6,1.2) --(7.8,1.2);
 	\draw (6,0) --(6,1.2);
 	\draw (7.8,0) --(7.8,1.2);
 	\node at (7.3,-0.55) {\scriptsize$B$};
 	%\node at (6.2,0.2) {$t$};
 	\draw[thick,dashed]  (6,0)-- (7.8,0);
 	%\node at (6.4,-.5) {$S^{\lambda^L}$};
 	%%%%%%%%%%C
 	\node at (9.9,0.3) {$s_2$};
 	\draw[thick] (8.5,1.2) --(10.3,1.2);
 	\draw (8.5,0) --(8.5,1.2);
 	\draw (10.3,0) --(10.3,1.2);
 	\node at (8.9,-.55) {\scriptsize$C$};
 	%\node at (10,0.2) {$t$};
 	\draw[thick,dashed]  (8.5,0)-- (10.3,0);
 	%\node at (10,-.5) {$S^{\lambda^R}$};
 	
 	%%%%%%%%%%AD
 	\node at (1.6,1) {\tiny\textbullet};
 		\node at (0.5,0.9) {\tiny\textbullet};
 	\draw[<->,dashed] (0.35,0.75) rectangle (0.65,1.05);
 	\node at (2.7,1) {\tiny\textbullet};
 	\draw (1.6,1) ..controls (2.15,0.8)..(2.7,1);
 	\node at (2.15,1.3) {\scriptsize$T_{AD}$};
 	
 	%%%%%%%%%%BC
 	\node at (7.6,1) {\tiny\textbullet};
 		\node at (6.35,0.9) {\tiny\textbullet};
 	\draw[<->,dashed] (6.2,0.75) rectangle (6.5,1.05);
 	\node at (8.7,1) {\tiny\textbullet};
 	\draw (7.6,1) ..controls (8.15,0.8)..(8.7,1);
 	\node at (8.15,1.3) {\scriptsize$T_{BC}$};
 	
 	%%%%%%%%%%AC
 	\node at (1,1) {\tiny\textbullet};
 	\node at (9.4,1) {\tiny\textbullet};
 		\node at (9.95,0.9) {\tiny\textbullet};
 	\draw[<->,dashed] (9.8,0.75) rectangle (10.1,1.05);
 	\draw (1,1) ..controls (2,3) and (8.2,3)..(9.4,1);
 	\node at (5.1,2.8) {\scriptsize$T_{AC}$};
 	
 	%%%%%%%%%%BD
 	\node at (3.4,1) {\tiny\textbullet};
 		\node at (3.9,0.9) {\tiny\textbullet};
 	\draw[<->,dashed] (3.75,0.75) rectangle (4.05,1.05);
 	\node at (6.9,1) {\tiny\textbullet};
 	\draw (3.4,1) ..controls (4.4,2) and (5.9,2)..(6.9,1);
 	\node at (5.1,2) {\scriptsize$T_{BD}$};
 	
 	\end{tikzpicture}
 	\end{array}
 	\end{equation*}
 	\vskip -0.3cm
 	\caption{Schematic presentation of Fig. \ref{moving}}
 	\label{schemoving}
 \end{figure}

Denote by $I_{(r_1|s_1,r_2|s_2)}^l$ the index set of
all tuples  $(T_{AC},T_{BD},T_{AD},T_{BC})$ of $(r_1|s_1,r_2|s_2)$-double-
walled diagrams with $l$ arcs (equally, the partial one-row $(r, s,l)$ diagrams), and
endow $I_{(r_1|s_1,r_2|s_2)}^l$ with a linear order by ordered lexicographically
with $T_{AC},T_{BD}$ using the order of the natural numbers, while $T_{AD},T_{BC}$ using the inverse order
of the natural numbers.

Let $r_1 + r_2=r$ and $s_1 + s_2=s$.
Suppose that $\textbf{t}=(T_{AC},T_{BD},T_{AD},T_{BC})$ is a possible tuple of an $(r_1|s_1,r_2|s_2)$-double-walled diagram
with $$T_{AC}+T_{BD}+T_{AD}+T_{BC}=l.$$ Define $\widetilde{V}_{(r_1|s_1,r_2|s_2)}^{\leq\textbf{t}}$ to
be the space spanned by all $(r_1|s_1,r_2|s_2)$-double-walled diagrams with $l$ arcs
such that  the indices in $I_{(r_1|s_1,r_2|s_2)}^l$ no more than $\textbf{t}$.
It is clearly a  $B_{r_1,s_1}(\de)\o B_{r_2,s_2}(\de)$-module.

\begin{defin}\label{wall}
The module $\widetilde{V}_{(r_1|s_1,r_2|s_2)}^{\leq\textbf{t}}$ is called an
$(r_1|s_1,r_2|s_2)$-double-walled module with the tuple $\textbf{t}$.
Moreover, let	$\lambda^L\vdash r-l, \lambda^R\vdash s-l$
and  $$W_{(r_1|s_1,r_2|s_2)}^{(\la^L,\la^R)}(\textbf{t}):=
\widetilde{V}_{(r_1|s_1,r_2|s_2) }^{\leq\textbf{t}} \o (S^{\la^L}\boxtimes S^{\la^R}),$$
which will be called a double-walled module with pair $(\la^L, \la^R)$.
\end{defin}

\noindent\textbf{Remarks}
(1)~~  Since each partial one-row  diagram can be viewed as a double walled diagram,
we replace partial one-row  diagrams with double walled diagrams and consequently use double-walled modules in our study.

\smallskip

(2)~~The module $W_{(r_1|s_1,r_2|s_2)}^{(\la^L,\la^R)}(\textbf{t})$ is a submodule of the restriction module
$$\res_{\scriptsize{B_{r_1,s_1}\o B_{r_2,s_2}}}^{\scriptsize{ B_{r,s}}}(\Delta_{r,s}(\la^L,\la^R)).$$
The fact that  $W_{(r_1|s_1,r_2|s_2)}^{(\la^L,\la^R)}(\textbf{t})$  is
actually a $(\scriptsize{B_{r_1, s_1}\o B_{r_2, s_2}})$-module is a direct corollary of Lemma \ref{updown}.

\smallskip

According to Definition \ref{wall}, we have a chain of modules indexed by $I_{(r_1|s_1,r_2|s_2)}^l$:

\begin{small}
$$0\subset \cdots \subset W_{(r_1|s_1,r_2|s_2)}^{(\la^L,\la^R)}
(\textbf{t}^\prime)\subset W_{(r_1|s_1,r_2|s_2)}^{(\la^L,\la^R)}(\textbf{t})\subset
\cdots \subset
\res_{\scriptsize{B_{r_1, s_1}\o B_{r_2,s_2}}}^{\scriptsize{ B_{r,s}}}(\Delta_{r,s}(\la^L,\la^R)). \eqno(*)$$
\end{small}

Now we are in a position to give the main result of this section,
which play an important role in the calculation shown in the next section.

\begin{thm}\label{keyp} Keep notations as above. Then
\begin{enumerate}
\item[\rm (1)]	Let $\De_{(r_1|s_1,r_2|s_2)}^{(\la^L,\la^R)}(\textbf{t})$ be a subquotient of the chain $(*)$. Then
	$$\De_{(r_1|s_1,r_2|s_2)}^{(\la^L,\la^R)}(\textbf{t}) \simeq
\widetilde{V}_{(r_1|s_1,r_2|s_2) }^{\textbf{t}} \o (S^{\la^L}\boxtimes S^{\la^R}),$$
where the space $\widetilde{V}_{(r_1|s_1,r_2|s_2)}^{\textbf{t}}$ is spanned by all
$(r_1|s_1,r_2|s_2)$-double walled diagrams with index   $\textbf{t}$.

\smallskip

\item[\rm (2)] Let $\lambda^L\vdash r-l$ and  $\lambda^R\vdash s-l$.	
Then, as a $(\scriptsize{B_{r_1,s_1}\o B_{r_2,s_2}})$-module,
	each summand of  $\De_{(r_1|s_1,r_2|s_2)}^{(\la^L,\la^R)}(\textbf{t})$ is filtered by cell modules of the form
	$$	
	\begin{array}{ll}
		V_{r_1, s_1}^{T_{AD}}\o \Big (~ \ind(\,S^{\la_{L}^1}\o S^{\mu_{1}}\,) \,
\boxtimes \, \ind(\,S^{\la_{R}^2}\o S^{\mu_{2}}\,)~\Big ) & \text{(left)} 		 \vspace{0.1cm} \\	\hspace{5cm}
	\bigotimes\\
		\vspace{0.8cm} V_{r_2, s_2}^{T_{BC}}\o \Big (~\ind(\,S^{\la_{L}^{2}}\o S^{\mu_{2}}\,) \,
\boxtimes \, \ind(\,S^{\la_{R}^{1}}\o S^{\mu_{1}}\,) ~\Big )& \text{(right)}
	\end{array},$$
	where
	\[
	\left\{
	\begin{array}{ll}
	\la_{L}^{1}\vdash ( r_1-T_{AD}-T_{AC}), &
	\mu_{1}\vdash T_{AC}, \\
	\la_{R}^{2}\vdash ( s_1-T_{AD}-T_{BD}),&\\
	\la_{L}^{2}\vdash ( r_2-T_{BC}-T_{BD}),&  \mu_{2}\vdash T_{BD},\\
   \la_{2}^{R}\vdash ( s_2-T_{BC}-T_{AC}).&
	\end{array}
	\right.
	\]
Here, $\ind (\,S^{\la}\o S^{\mu}\,)$
means an induction   from   $k\Sigma_{r}\o k\Sigma_{s} $
	to $k\Sigma_{r+s}$. Note that we will omit the subscripts of the functor $\ind$ if there is no dangerous of being ambiguous.

\smallskip

\item[\rm (3)] The  chain $(\ast)$ after being refined by {\rm (2)}
	is a composition series of the restricted module
$\res_{\scriptsize{B_{r_1, s_1}\o B_{r_2,s_2}}}^{\scriptsize{ B_{r,s}}}(\Delta_{r,s}(\la^L,\la^R))$.
\end{enumerate}	
\end{thm}

\begin{proof}
The Statement (1) follows directly from the definition of walled modules, and
(3) can be obtained immediately once (2) has been proved since the chain $(\ast)$ is filtered by cell modules.
As a result, we only need to prove (2).
Let us first consider two special cases, which are important for dealing with the general case.

\smallskip

\textbf{Special case 1:}  $\mathbf{r_1=s_2=T_{AC}},$ \textbf{and} $\mathbf{r_2=s_1=T_{BD}}$

\smallskip

In this case, the index $\mathbf{t=(T_{AC},T_{BD},T_{AD},T_{BC})=(r_1,r_2,0,0)}$.
Given a double walled diagram $\tilde{v}$
with index $(r_1,r_2,0,0)$, it can be identified with two permutations
$\sigma \in \Sigma_{T_{AC}}$ and $\sigma^{\prime} \in \Sigma_{T_{BD}}$
(see Fig.\ref{S11} for an example).

\begin{figure}[H]
	\begin{equation*}
		\ytableausetup{boxsize=0.8em}
	\begin{array}{c}
\begin{tikzpicture}
%	\node at (0,1.4) {$1$};
%	\node at (0.8,1.4) {$3$};
%	\node at (0.42,1.4) {$2$};
%	\node at (1.6,1.4) {$4$};
%	\node at (2.4,1.4) {$\cdots$};
%	\node at (3.2,1.4) {$8$};
%\draw [thick,dotted] (-0.9,1.2) --(-0.9,1.9);
%\draw [thick,dotted] (-0.9,1.2) --(1.6,1.2);
%\draw [thick,dotted] (1.6,1.9) --(1.6,1.2);
%\draw [thick,dotted] (1.6,1.9) --(-0.9,1.9);
%\node at (-0.3,1.5) {\scriptsize I};
%\node at (0.7,1.5) {\scriptsize IV};
\node at (-0.4,0.9) {\tiny\textbullet};
\node at (-0.4,1.2) {\tiny $1$};
\node at (-0.4,1.54) {\scriptsize $\sigma=(21)$};
\draw (-0.4,0.9) .. controls (0.1,-0.2) and (3.1,-0.20) .. (4,0.9);
%\node at (-0.8,0.9) {\tiny\textbullet};
%\draw (-0.8,0.4) --(-0.8,0.9);
\node at (0,0.9) {\tiny\textbullet};
\node at (0,1.2) {\tiny $2$};
\draw (0,0.9) .. controls (0.9,-0.15) and (3.9,-0.1) .. (4.4,0.9);
\draw [thick] (0.2,0) --(0.2,1.3);
\node at (0.4,0.9) {\tiny\textbullet};
\node at (0.4,1.24) {\tiny $1^{\prime}$};
\draw (0.4,0.9) .. controls (0.9,0.1) and (2.8,0.1) .. (3.2,0.9);
%\draw (0.4,0.9) .. controls (0.8,0.4) and (2.8,0.4) .. (3.2,0.9);
\node at (0.8,0.9) {\tiny\textbullet};
\node at (0.8,1.24) {\tiny $2^{\prime}$};
\node at (1.1,1.54) {\scriptsize $\sigma^{\prime}=(2^{\prime}3^{\prime}1^{\prime})$};
\draw (0.8,0.9) .. controls (1.2,0.3) and (2.4,0.3) .. (2.8,0.9);
\node at (1.2,0.9) {\tiny\textbullet};
\node at (1.2,1.24) {\tiny $3^{\prime}$};

%\node at (1.6,0.9) {\tiny\textbullet};
%\draw (1.6,0.9) --(1.6,0.4);
\draw[thick] (3.8,1.3) --(3.8,0);
\draw [thick,dashed] (2,0) --(2,1.3);

\node at (2.8,0.9) {\tiny\textbullet};
\node at (2.8,1.24) {\tiny $3^{\prime}$};

\node at (3.2,0.9) {\tiny\textbullet};
\node at (3.2,1.24) {\tiny $2^{\prime}$};
\draw (1.2,0.9) .. controls (1.8,0.6) and (3,0.6) .. (3.6,0.9);
\node at (3.6,0.9) {\tiny\textbullet};	
\node at (3.6,1.24) {\tiny $1^{\prime}$};
\node at (4,0.9) {\tiny\textbullet};
\node at (4,1.2) {\tiny $2$};
\node at (4.4,0.9) {\tiny\textbullet};
\node at (4.4,1.2) {\tiny $1$};

%\draw (4.4,0.9) --(4.4,0.4);
%\draw [thick,dotted] (2.5,1.2) --(2.5,1.9);
%\draw [thick,dotted] (2.5,1.2) --(4.5,1.2);
%\draw [thick,dotted] (4.5,1.9) --(4.5,1.2);
%\draw [thick,dotted] (4.5,1.9) --(2.5,1.9);
%\node at (3,1.5) {\scriptsize II};
%\node at (3.8,1.5) {\scriptsize III};

\node at (0.7,0.1) {\scriptsize D};
\node at (-0.3,0.1) {\scriptsize A};
\node at (3.2,0.1) {\scriptsize B};
\node at (4.2,0.1) {\scriptsize C};
\end{tikzpicture}
\end{array}
\end{equation*}
\caption{Permutations determined by arcs}
\label{S11}
\end{figure}		

%%%%%%%%%%%%%%A
Then as a module over
$k(\Sigma_{T_{AC}}\times \Sigma_{T_{BD}})\times k(\Sigma_{T_{AC}}\times \Sigma_{T_{BD}})$,
the double walled module is isomorphic to $k(\Sigma_{T_{AC}}\times \Sigma_{T_{BD}})$,
where the action is given by
$$\big((g_1, g_1^{\prime}),(g_2, g_2^{\prime})\big)\cdot (\sigma, \sigma^{\prime})
=\big(g_1\sigma g_2^{-1}\,,\, g_1^{\prime}\sigma^{\prime} (g_2^{\prime})^{-1}\big)$$
for $g_1, g_2 \in \Sigma_{T_{AC}}$, and $g_1^{\prime}, g_2^{\prime} \in \Sigma_{T_{BD}} $.
Consequently, by using the cellular structure
of $k(\Sigma_{T_{AC}}\times \Sigma_{T_{BD}})$
the module has a filtration with subquotients
$$(S^{\mu_{1}}\boxtimes S^{\mu_{2}}) \o (S^{\mu_{2}}\boxtimes S^{\mu_{1}})$$
where $\mu_{1}\vdash T_{AC},\, \mu_{2}\vdash T_{BD}$. Clearly, the action of each element of
$\scriptsize{B_{r_1,s_1}\o B_{r_2,s_2}}$ containing at least one arc on subquotients is zero.
This implies that as a $\scriptsize{B_{r_1,s_1}\o B_{r_2,s_2}}$-module,
$\De_{(r_1|s_1,r_2|s_2)}^{(\la^L,\la^R)}(\textbf{t})$ is filtered by modules of the form:
$$
\begin{array}{ll}
\big(V^{0}_{r_1,s_1}\o (S^{\mu_{1}}\boxtimes S^{\mu_{2}}) \big)\, \bigotimes\,
\big(V^{0}_{r_2,s_2}\o (S^{\mu_{1}}\boxtimes S^{\mu_{2}}) \big).
\end{array}$$

We draw a schematic diagram to  help comprehend the case ( see Fig. \ref{S1}).
\begin{figure}[t]
\begin{equation*}
\begin{array}{c}
\begin{tikzpicture}[scale=0.65]
\node at (0.3,0.9) {$r_1$};
\draw[thick] (0,1.2) --(1.8,1.2);
\draw (0,0) --(0,1.2);
\draw (1.8,0) --(1.8,1.2);
\node at (1.3,0.55) {\scriptsize$A$};
%\node at (0.2,0.2) {$t$};
\draw[thick,dashed]  (0,0) --(1.8,0);
%\node at (0.4,-.5) {$S^{\lambda^L}$};
%%%%%%%%%%D
\node at (3.9,0.9) {$s_2$};
\draw[thick] (2.5,1.2) --(4.3,1.2);
\draw (2.5,0) --(2.5,1.2);
\draw (4.3,0) --(4.3,1.2);
\node at (2.9,0.55) {\scriptsize$D$};
%\node at (4,0.2) {$t$};
\draw[thick,dashed]  (2.5,0)-- (4.3,0);
%\node at (4,-.5) {$S^{\lambda^R}$};

%%%%%%%%%%%%%%B
\node at (6.3,0.9) {$r_2$};
\draw[thick] (6,1.2) --(7.8,1.2);
\draw (6,0) --(6,1.2);
\draw (7.8,0) --(7.8,1.2);
\node at (7.3,0.55) {\scriptsize$B$};
%\node at (6.2,0.2) {$t$};
\draw[thick,dashed]  (6,0) --(7.8,0);
%\node at (6.4,-.5) {$S^{\lambda^L}$};
%%%%%%%%%%C
\node at (9.9,0.9) {$s_1$};
\draw[thick] (8.5,1.2) --(10.3,1.2);
\draw (8.5,0) --(8.5,1.2);
\draw (10.3,0) --(10.3,1.2);
\node at (8.9,.55) {\scriptsize$C$};
%\node at (10,0.2) {$t$};
\draw[thick,dashed]  (8.5,0) -- (10.3,0);
%\node at (10,-.5) {$S^{\lambda^R}$};

%%%%%%%%%%AD
%\node at (1.6,1) {\tiny\textbullet};
%\node at (2.7,1) {\tiny\textbullet};
%\draw (1.6,1) ..controls (2.15,0.8)..(2.7,1);
%\node at (2.15,1.3) {\scriptsize$T_{AD}$};
	\draw [thick] (2.18,-0.2) --(2.18,1.4);
%%%%%%%%%%%BC
%\node at (7.6,1) {\tiny\textbullet};
%\node at (8.7,1) {\tiny\textbullet};
%\draw (7.6,1) ..controls (8.15,0.8)..(8.7,1);
%\node at (8.15,1.3) {\scriptsize$T_{BC}$};
	\draw [thick] (8.18,-0.2) --(8.18,1.4);
%%%%%%%%%%AC
\node at (0.8,1) {\tiny\textbullet};
\node at (9.4,1) {\tiny\textbullet};
\draw (0.8,1) ..controls (2,3) and (8.2,3)..(9.4,1);
\node at (5.1,2.85) {
$\Sigma_{T_{AC}}$
};
%%%%%%%%%%BD
\node at (3.4,1) {\tiny\textbullet};
\node at (6.9,1) {\tiny\textbullet};
\draw (3.4,1) ..controls (4.4,1.5) and (5.9,1.5)..(6.9,1);
\node at (5.1,1.7) {$\Sigma_{T_{BD}}$};

\end{tikzpicture}
\end{array}
\small {\rightsquigarrow}
	\begin{array}{c}
	\begin{tikzpicture}[scale=0.65]
	
	%%%%%%%%%%%%%%A
	\node at (0.3,0.9) {$r_1$};
	\draw[thick] (0,1.2) --(1.8,1.2);
	\draw (0,0) --(0,1.2);
	\draw (1.8,0) --(1.8,1.2);
	\node at (1,0.3) {\scriptsize$A$};
	%\node at (0.2,0.2) {$t$};
	\draw[thick,dashed]  (0,0) --(1.8,0);
	%\node at (0.4,-.5) {$S^{\lambda^L}$};
	%%%%%%%%%%D
	\node at (3.9,0.9) {$s_2$};
	\draw[thick] (2.5,1.2) --(4.3,1.2);
	\draw (2.5,0) --(2.5,1.2);
	\draw (4.3,0) --(4.3,1.2);
	
%		\node at (4.3,1) {\scriptsize$S^{\mu_{1}}$};
	\node at (3.3,0.3) {\scriptsize$D$};
	%\node at (4,0.2) {$t$};
	\draw[thick,dashed]  (2.5,0)-- (4.3,0);
	%\node at (4,-.5) {$S^{\lambda^R}$};

	%%%%%%%%%%%%%%B
	\node at (6.3,0.9) {$r_2$};
	\draw[thick] (6,1.2) --(7.8,1.2);
	\draw (6,0) --(6,1.2);
	\draw (7.8,0) --(7.8,1.2);
	\node at (7,0.3) {\scriptsize$B$};
	%\node at (6.2,0.2) {$t$};
	\draw[thick,dashed]  (6,0) --(7.8,0);
	%\node at (6.4,-.5) {$S^{\lambda^L}$};
	%%%%%%%%%%C
	\node at (9.9,0.9) {$s_1$};
	\draw[thick] (8.5,1.2) --(10.3,1.2);
	\draw (8.5,0) --(8.5,1.2);
	\draw (10.3,0) --(10.3,1.2);
	\node at (9.2,.3) {\scriptsize$C$};
	%\node at (10,0.2) {$t$};
	\draw[thick,dashed]  (8.5,0) -- (10.3,0);
	%\node at (10,-.5) {$S^{\lambda^R}$};
	
	%%%%%%%%%%AD
	%\node at (1.6,1) {\tiny\textbullet};
	%\node at (2.7,1) {\tiny\textbullet};
	%\draw (1.6,1) ..controls (2.15,0.8)..(2.7,1);
	%\node at (2.15,1.3) {\scriptsize$T_{AD}$};
	\draw [thick] (2.18,-0.2) --(2.18,1.4);
	%%%%%%%%%%%BC
	%\node at (7.6,1) {\tiny\textbullet};
	%\node at (8.7,1) {\tiny\textbullet};
	%\draw (7.6,1) ..controls (8.15,0.8)..(8.7,1);
	%\node at (8.15,1.3) {\scriptsize$T_{BC}$};
		\draw [thick] (8.18,-0.2) --(8.18,1.4);
	%%%%%%%%%%AC
	\node at (0.8,0.9) {\tiny\textbullet};
		\draw[<->,dashed] (0.8,0.9) circle (0.18);
	\node at (9.4,0.9) {\tiny\textbullet};
		\draw[<->,dashed] (9.4,0.9) circle (0.18);
		\draw (9.4,1) --(9.4,1.6);
			\node at (9.5,1.8) {\scriptsize$S^{\mu_{1}}$};
	\draw (0.8,1) --(0.8,1.6);
		\node at (0.9,1.8) {\scriptsize$S^{\mu_{1}}$};
%	\node at (5.1,2.8) {\scriptsize$T_{AC}$};
	
	%%%%%%%%%%BD
	\node at (3.4,0.9) {\tiny\textbullet};
	\draw[<->,dashed] (3.4,0.9) circle (0.18);
		\draw (3.4,1) --(3.4,1.6);
			\node at (3.5,1.8) {\scriptsize$S^{\mu_{2}}$};
	\node at (6.9,0.9) {\tiny\textbullet};
		\draw[<->,dashed] (6.9,0.9) circle (0.18);
		\draw (6.9,1) --(6.9,1.6);
			\node at (7,1.8) {\scriptsize$S^{\mu_{2}}$};
%	\draw (3.4,1) ..controls (4.4,1.2) and (5.9,1.22)..(6.9,1);
%	\node at (5.1,2) {\scriptsize$T_{BD}$};
	\draw [thick,dashed] (5,0) --(5,1.3);
		\node at (0.8,3.6) {};
			\node at (3.92,-0.77) {{\tiny\textbullet} - new free dots  in  restriction};
			\draw[<->,dashed] (0.46,-0.845) circle (0.18);
	\end{tikzpicture}
	\end{array}
	\end{equation*}
	\caption{Demonstrations for restriction of Special case 1}
	\label{S1}
\end{figure}
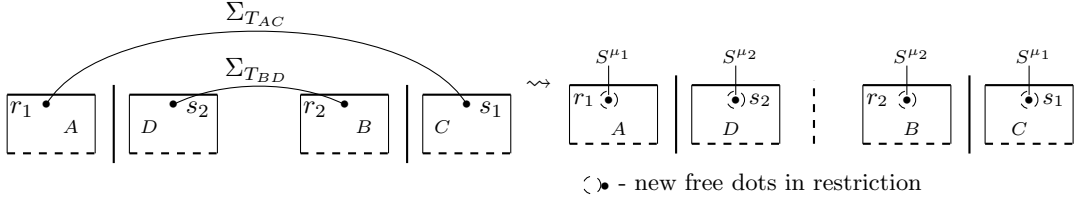

\textbf{Special case 2:} $\mathbf{T_{AC}=0,T_{BD}=0}$

\smallskip

Note that $S^{\la_L}\boxtimes S^{\la_R}$, in the double walled module
$\widetilde{V}_{(r_1|s_1,r_2|s_2) }^{\textbf{t}} \o (S^{\la_L}\boxtimes S^{\la_R})$,
corresponds to the free dots of diagrams.
This is shown in
Fig. \ref{S2} (a) with $S^{\la_L}$ corresponding to the free dots of block $A$ jointed $B$ as  $S^{\la_R}$ corresponding to free dots of block $C$ jointed $D$. Then $\la_L\vdash r-T_{AD}-T_{BC}, \la_R\vdash s-T_{AD}-T_{BC}$.
Therefore, from the characteristic-free version of the Littlewood-Richardson Rule [13], we have
that the restricted module $S^{\la_L}\downarrow^{\Sigma_{r-T_{AD}-T_{BC}}}_{\Sigma_{r_1-T_{AD} }\,\times\, \Sigma_{r_2-T_{BC} }}$
has a filtration with subquotients  $S^{\la_L^1}\boxtimes S^{\la_L^2}$, where $S^{\la_L^1}$ and $S^{\la_L^2}$ are shown  in
Fig.\ref{S2} (b). Let us give some explain here. Module
$S^{\la_L^1}$ corresponds to the free dots of block $A$ on the left side of the wall
with $\la_L^1\vdash r_1-T_{AD}$, as $S^{\la_L^2}$ corresponds to the
right free dots of block $B$ with $ \la_L^2\vdash r_2-T_{BC}$.
Similarly, we can study the restriction of $S^{\la_R}$.
Note that any diagram of $\scriptsize{B_{r_1,s_1}}$ with strictly more
than $T_{AD}$ arcs will kill the summand, as does any diagram
of $\scriptsize{B_{r_2,s_2}}$ with strictly more than $T_{BC}$ arcs.
Hence, as a $\scriptsize{B_{r_1,s_1}\o B_{r_2,s_2}}$-module,
$\De_{(r_1|s_1,r_2|s_2)}^{(\la^L,\la^R)}(\textbf{t})$ is filtered by modules of the form:
$$
\begin{array}{ll}
\big(V^{T_{AD}}_{r_1,s_1}\o (S^{\la_{L}^1}\boxtimes S^{\la_{R}^{2}}) \big)\,
\bigotimes\, \big(V^{T_{BC}}_{r_2,s_2}\o (S^{\la_{L}^2}\boxtimes S^{\la_{R}^{1}}) \big).
\end{array}$$

\begin{figure}[H]
	\begin{equation*}
	\begin{array}{c}
	\begin{tikzpicture}[scale=0.7]
	%%%%%%A
	\node at (-1.5,0.3) {$r_1$};
	\draw[thick] (-1.8,1.2) --(0,1.2);
	\draw (-1.8,0) --(-1.8,1.2);
	\draw[dashed] (0,0) --(0,1.2);
	\node at (-0.8,-0.4) {\scriptsize$A$};
	%\node at (-1.6,0.2) {$t$};
	\draw[thick,dashed]  (-1.8,0)-- (0,0);
	%\node at (-1.4,-.5) {$S^{\lambda^L}$};
	
	%%%%%%%%%%%%%%B
	\node at (0.3,0.9) {\tiny\textbullet};
	\draw[thick] (0,1.2) --(1.8,1.2);
	\draw[dashed] (0,0) --(0,1.2);
	\draw (1.8,0) --(1.8,1.2);
	\node at (1,-0.4) {\scriptsize$B$};
	%\node at (0.2,0.2) {$t$};
	\draw[thick,dashed]  (0,0) -- (1.8,0);
	%\node at (0.4,-.5) {$S^{\lambda^L}$};
	%%%%%%%%%%C
	\node at (4,0.9) {\tiny\textbullet};
	\draw[thick] (2.5,1.2) --(4.3,1.2);
	\draw (2.5,0) --(2.5,1.2);
	\draw[dashed] (4.3,0) --(4.3,1.2);
	\node at (3.4,-0.4) {\scriptsize$C$};
	%\node at (4,0.2) {$t$};
	\draw[thick,dashed]  (2.5,0) -- (4.3,0);
	%\node at (4,-.5) {$S^{\lambda^R}$};
	%%%%%%%%%%%%%%%%%D
	\node at (5.8,0.3) {$s_2$};
	\draw[thick] (4.3,1.2) --(6.1,1.2);
	\draw[dashed] (4.3,0) --(4.3,1.2);
	\draw (6.1,0) --(6.1,1.2);
	\node at (5.1,-0.4) {\scriptsize$D$};
	%\node at (5.8,0.2) {$t$};
	\draw[thick,dashed]  (4.3,0) --(6.1,0);
	%\node at (5.8,-.5) {$S^{\lambda^R}$};
	%%%%%%%%%%BC
	\node at (1.6,1) {\tiny\textbullet};
	\node at (2.7,1) {\tiny\textbullet};
	\draw (1.6,1) ..controls (2.15,0.8)..(2.7,1);
	\node at (2.15,1.4) {\scriptsize$T_{BC}$};
		\draw[thick]  (2.15,0)-- (2.15,1.2);
	%%%%%%%%%%AD
	\node at (-1,1) {\tiny\textbullet};
	\node at (5.3,1) {\tiny\textbullet};
	\draw (-1,1) ..controls (-0.6,3.5) and (4.9,3.5)..(5.3,1);
	\node at (2.15,3.1) {\scriptsize$T_{AD}$};
	%%%%%%%%%%AC
	\node at (-0.3,0.9) {\tiny\textbullet};
	\node at (3,0.3) {$s_1$};
%	\draw (-0.3,1) ..controls (0.4,2) and (2.7,2)..(3.4,1);
%	\node at (1.55,2) {\scriptsize$T_{AC}$};
	%%%%%%%%%%BD
	\node at (1.3,0.3) {$r_2$};
	\node at (4.6,0.9) {\tiny\textbullet};
%	\draw (1,1) ..controls (1.7,2) and (4,2)..(4.7,1);
%	\node at (2.9,2) {\scriptsize$T_{BD}$};
\draw[thick,<->,dashed] (-0.5,0.7) rectangle (-0.1,1.1);
\draw[thick,<->,dashed] (0.1,0.7) rectangle (0.5,1.1);
\draw (-0.3,0.9) -- (0,1.55);
\draw (0.3,0.9) -- (0,1.55);
\node at (0.1,1.8) {$S^{\la_{L}}$};
\draw[thick,<->,dashed] (3.8,0.7) rectangle (4.2 ,1.1);
\draw[thick,<->,dashed] (4.4,0.7) rectangle (4.8 ,1.1);
\draw (4,0.9) -- (4.3,1.55);
\draw (4.6,0.9) -- (4.3,1.55);
\node at (4.3,1.8) {$S^{\la_{R}}$};
\node at (2.2,-1.1) { \small \mbox{(a) partial diagram module over $\scriptsize{B_{r,s}}$ }};
	\end{tikzpicture}
	\end{array}
\rightsquigarrow
	\begin{array}{c}
	\begin{tikzpicture}[scale=0.65]
	
	%%%%%%%%%%%%%%A
	\node at (0.3,0.3) {$r_1$};
	\draw[thick] (0,1.2) --(1.8,1.2);
	\draw (0,0) --(0,1.2);
	\draw (1.8,0) --(1.8,1.2);
	\node at (1,-0.4) {\scriptsize$A$};
	%\node at (0.2,0.2) {$t$};
	\draw[thick,dashed]  (0,0) --(1.8,0);
	%\node at (0.4,-.5) {$S^{\lambda^L}$};
	%%%%%%%%%%D
	\node at (3.9,0.3) {$s_2$};
	\draw[thick] (2.5,1.2) --(4.3,1.2);
	\draw (2.5,0) --(2.5,1.2);
	\draw (4.3,0) --(4.3,1.2);
	\node at (3.3,-0.4) {\scriptsize$D$};
	%\node at (4,0.2) {$t$};
	\draw[thick,dashed]  (2.5,0) --(4.3,0);
	%\node at (4,-.5) {$S^{\lambda^R}$};

	%%%%%%%%%%%%%%B
	\node at (6.3,0.3) {$r_2$};
	\draw[thick] (6,1.2) --(7.8,1.2);
	\draw (6,0) --(6,1.2);
	\draw (7.8,0) --(7.8,1.2);
	\node at (7,-0.4) {\scriptsize$B$};
	%\node at (6.2,0.2) {$t$};
	\draw[thick,dashed]  (6,0)-- (7.8,0);
	%\node at (6.4,-.5) {$S^{\lambda^L}$};
	%%%%%%%%%%C
	\node at (9.9,0.3) {$s_1$};
	\draw[thick] (8.5,1.2) --(10.3,1.2);
	\draw (8.5,0) --(8.5,1.2);
	\draw (10.3,0) --(10.3,1.2);
	\node at (9.4,-.4) {\scriptsize$C$};
	%\node at (10,0.2) {$t$};
	\draw[thick,dashed]  (8.5,0)-- (10.3,0);
	%\node at (10,-.5) {$S^{\lambda^R}$};
	
	%%%%%%%%%%AD
	\node at (1.6,1) {\tiny\textbullet};
	\node at (2.7,1) {\tiny\textbullet};
	\draw (1.6,1) ..controls (2.15,0.8)..(2.7,1);
	\node at (2.15,1.4) {\scriptsize$T_{AD}$};
		\draw[thick]  (2.15,0)-- (2.15,1.2);
	%%%%%%%%%%BC
	\node at (7.6,1) {\tiny\textbullet};
	\node at (8.7,1) {\tiny\textbullet};
	\draw (7.6,1) ..controls (8.15,0.8)..(8.7,1);
	\node at (8.15,1.4) {\scriptsize$T_{BC}$};
	\draw[thick]  (8.15,0)-- (8.15,1.2);
	%%%%%%%%%%AC
	\node at (0.8,0.9) {\tiny\textbullet};
	\draw[thick,<->,dashed] (0.6,0.7) rectangle (1,1.1);
		\draw (0.8,1) -- (0.8,1.55);
		\node at (1.1,1.9) {$S^{\la_{L}^{1}}$};
	\node at (9.4,0.9) {\tiny\textbullet};
		\draw[thick,<->,dashed] (9.2,0.7) rectangle (9.6,1.1);
			\draw (9.4,1) -- (9.4,1.55);
		\node at (9.4,1.9) {$S^{\la_{R}^{1}}$};
%	\draw (0.8,1) ..controls (2,3) and (8.2,3)..(9.4,1);
%	\node at (5.1,2.8) {\scriptsize$T_{AC}$};
	
	%%%%%%%%%%BD
	\node at (3.4,0.9) {\tiny\textbullet};
		\draw[thick,<->,dashed] (3.2,0.7) rectangle (3.6,1.1);
			\draw (3.4,1) -- (3.4,1.55);
		\node at (3.4,1.9) {$S^{\la_{R}^{2}}$};
	\node at (6.9,0.9) {\tiny\textbullet};
		\draw[thick,<->,dashed] (6.7,0.7) rectangle (7.1,1.1);
			\draw (6.9,1) -- (6.9,1.55);
		\node at (6.9,1.9) {$S^{\la_{L}^{2}}$};
%	\draw (3.4,1) ..controls (4.4,2) and (5.9,2)..(6.9,1);
%	\node at (5.1,2) {\scriptsize$T_{BD}$};
\draw[thick,dashed]  (5.1,0) --(5.1,1.3);
		\node at (5.3,-1.4) { \small \mbox{(b)  double walled module over $\scriptsize{B_{r_1,s_1}\o B_{r_2,s_2}}$ }};
		\node at (5,3.33) {};
	\end{tikzpicture}
	\end{array}
	\end{equation*}
	\vskip -0.3cm
	\caption{Demonstration for restriction of Special case 2}
	\label{S2}
\end{figure}
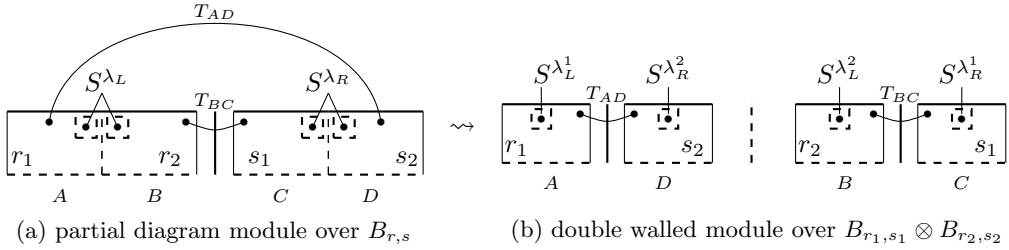
\textbf{General case.}
Given a double-walled diagram $\tilde{v} $ with index $(T_{AC},T_{BD},T_{AD},T_{BC})$,
it can be divided into three different parts:
\begin{enumerate}
\item[1)] the arcs jointing  blocks $AD$ or $BC$;

\item[2)] the arcs jointing blocks $AC$ or $BD$;

\item[3)] and free dots (jointing no arcs).
\end{enumerate}
As our study of the special cases above,  when we  consider restricting to
$\scriptsize{B_{r_1,s_1}\o B_{r_2,s_2}}$,  part (1)
determine the half diagram module $V^{T_{AD}}_{r_1,s_1}$,
$V^{T_{BC}}_{r_2,s_2}$, respectively;
part (2) determine new free dots (with dashed circle in Fig.\ref{gene}) corresponding Specht modules
 $S^{\mu_{1}}, S^{\mu_{2}}$ as special case 1;  for part (3),  as special case 2,
old free dots  (with dashed frame) correspond  Specht modules $S^{\la_{L}^{1}}, S^{\la_{L}^{2}}, S^{\la_{R}^{1}}$ and $S^{\la_{R}^{2}}$. Note that if a block, say $A$, contains both $r_1-T_{AB}$ old and $T_{AC}$ new free dots, then the Specht modules corresponding the free dots should be a $k\Sigma_{(r_1-T_{AB})+T_{AC}}$ module,  which  is  the induction $\ind_{k\Sigma_{(r_1-T_{AB})}\times k\Sigma_{T_{AC}}}^{k\Sigma_{(r_1-T_{AB})+T_{AC}}}(\,S^{\la_{L}^1}\o S^{\mu_{1}}\,)$.
Then the double-walled module is filtered by modules of the form shown in the theorem,
as shown in  Fig. \ref{gene}.
\end{proof}
%	$$	
%\begin{array}{ll}
%V_{r_1, s_1}^{T_{AD}}\o \Big (~ \ind(\,S^{\la_{L}^1}\o S^{\mu_{1}}\,) \, \boxtimes \, \ind(\,S^{\la_{R}^2}\o S^{\mu_{2}}\,)~\Big ) & \text{(left)} 	
%\end{array}
%$$
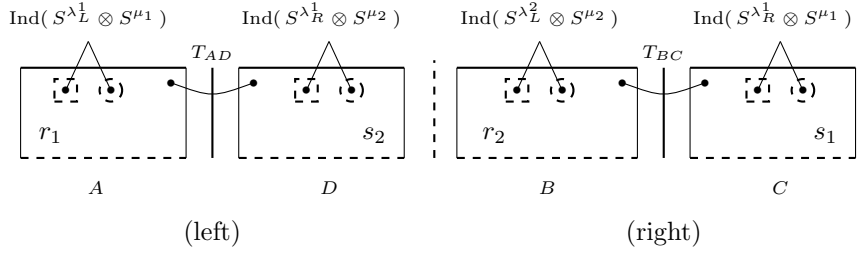
\begin{figure}[H]
	\begin{equation*}
	\begin{array}{c}
	\begin{tikzpicture}[scale=1]
	%%%%%%%%%%%%%%A
	\node at (0,0.3) {$r_1$};
	\draw[thick] (-0.4,1.2) --(1.8,1.2);
	\draw (-0.4,0) --(-0.4,1.2);
	\draw (1.8,0) --(1.8,1.2);
	\node at (0.6,-0.4) {\scriptsize$A$};
	%\node at (0.2,0.2) {$t$};
	\draw[thick,dashed]  (-0.4,0) --(1.8,0);
	%\node at (0.4,-.5) {$S^{\lambda^L}$};
	%%%%%%%%%%D
	\node at (4.3,0.3) {$s_2$};
	\draw[thick] (2.5,1.2) --(4.7,1.2);
	\draw (2.5,0) --(2.5,1.2);
	\draw (4.7,0) --(4.7,1.2);
	\node at (3.7,-0.4) {\scriptsize$D$};
	%\node at (4,0.2) {$t$};
	\draw[thick,dashed]  (2.5,0) --(4.7,0);
	%\node at (4,-.5) {$S^{\lambda^R}$};

	%%%%%%%%%%%%%%B
	\node at (5.9,0.3) {$r_2$};
	\draw[thick] (5.4,1.2) --(7.8,1.2);
	\draw (5.4,0) --(5.4,1.2);
	\draw (7.8,0) --(7.8,1.2);
	\node at (6.6,-0.4) {\scriptsize$B$};
	%\node at (6.2,0.2) {$t$};
	\draw[thick,dashed]  (5.4,0)-- (7.8,0);
	%\node at (6.4,-.5) {$S^{\lambda^L}$};
	%%%%%%%%%%C
	\node at (10.3,0.3) {$s_1$};
	\draw[thick] (8.5,1.2) --(10.7,1.2);
	\draw (8.5,0) --(8.5,1.2);
	\draw (10.7,0) --(10.7,1.2);
	\node at (9.7,-.4) {\scriptsize$C$};
	%\node at (10,0.2) {$t$};
	\draw[thick,dashed]  (8.5,0)-- (10.7,0);
	%\node at (10,-.5) {$S^{\lambda^R}$};
	
	%%%%%%%%%%AD
	\node at (1.6,1) {\tiny\textbullet};
	\node at (2.7,1) {\tiny\textbullet};
	\draw (1.6,1) ..controls (2.15,0.8)..(2.7,1);
	\node at (2.15,1.4) {\scriptsize$T_{AD}$};
		\node at (2.15,-1) {(left)};
	\draw[thick]  (2.15,0)-- (2.15,1.2);
	%%%%%%%%%%BC
	\node at (7.6,1) {\tiny\textbullet};
	\node at (8.7,1) {\tiny\textbullet};
	\draw (7.6,1) ..controls (8.15,0.8)..(8.7,1);
	\node at (8.15,1.4) {\scriptsize$T_{BC}$};
	\node at (8.15,-1) {(right)};
	\draw[thick]  (8.15,0)-- (8.15,1.2);
		\node at (6.2,0.9) {\tiny\textbullet};
	\draw[thick,<->,dashed] (6.05,0.75) rectangle (6.35,1.05);
	\node at (6.8,0.9) {\tiny\textbullet};
	\draw[thick,<->,dashed] (6.8,0.9) circle (0.15);
		\node at (6.5,1.9) {\scriptsize$\ind(\,S^{\la_{L}^2}\o S^{\mu_{2}}\,)$};
	\draw (6.2,0.9) -- (6.5,1.55);
	\draw (6.8,0.9) -- (6.5,1.55);
	%%%%%%%%%%AC
%	\node at (0.8,0.9) {\tiny\textbullet};
%	\draw[thick,<->,dashed] (0.6,0.7) rectangle (1,1.1);
	\node at (0.8,0.9) {\tiny\textbullet};
%\draw[<->,dashed] (0.65,0.75) rectangle (0.95,1.05);
\draw[thick,<->,dashed] (0.8,0.9) circle (0.15);

\node at (0.2,0.9) {\tiny\textbullet};
\draw[thick,<->,dashed] (0.05,0.75) rectangle (0.35,1.05);
%	\draw (0.8,1) ..controls (2,3) and (8.2,3)..(9.4,1);
%	\node at (5.1,2.8) {\scriptsize$T_{AC}$};
%\draw[thick,<->,dashed] (9.1,0.6) rectangle (10.3,1.14);
		\node at (0.5,1.9) {\scriptsize$\ind(\,S^{\la_{L}^1}\o S^{\mu_{1}}\,)$};
	\draw (0.8,0.9) -- (0.5,1.55);
	\draw (0.2,0.9) -- (0.5,1.55);
	\node at (9.4,0.9) {\tiny\textbullet};
	\draw[thick,<->,dashed] (9.25,0.75) rectangle (9.55,1.05);
			\node at (10,0.9) {\tiny\textbullet};
%		\draw[<->,dashed] (9.85,0.75) rectangle (10.15,1.05);
	\draw[thick,<->,dashed] (10,0.9) circle (0.15);
	\node at (9.6,1.9) {\scriptsize$\ind(\,S^{\la_{R}^1}\o S^{\mu_{1}}\,)$};
	\draw (10,0.9) -- (9.7,1.55);
	\draw (9.4,0.9) -- (9.7,1.55);
	%	\draw (0.8,1) ..controls (2,3) and (8.2,3)..(9.4,1);
	%	\node at (5.1,2.8) {\scriptsize$T_{AC}$};
%		\draw[thick,<->,dashed] (9.1,0.6) rectangle (10.3,1.14);
	%%%%%%%%%%BD
%	\node at (3.4,0.9) {\tiny\textbullet};
	\node at (3.4,0.9) {\tiny\textbullet};
\draw[thick,<->,dashed] (3.25,0.75) rectangle (3.55,1.05);
%\draw (9.7,1.2) -- (9.7,1.55);
\node at (4,0.9) {\tiny\textbullet};
%\draw[<->,dashed] (3.85,0.75) rectangle (4.15,1.05);
\draw[thick,<->,dashed] (4,0.9) circle (0.15);
%	\draw[thick,<->,dashed] (3.2,0.7) rectangle (3.6,1.1);
%	\draw (3.4,1) -- (3.4,1.55);
%	\node at (3.4,1.9) {$S^{\la_{R}^{2}}$};
		\node at (3.6,1.9) {\scriptsize$\ind(\,S^{\la_{R}^1}\o S^{\mu_{2}}\,)$};
	\draw (3.4,0.9) -- (3.7,1.55);
	\draw (4,0.9) -- (3.7,1.55);
%	\node at (6.9,0.9) {\tiny\textbullet};
%	\draw (6.9,1) -- (6.9,1.55);
%	\node at (6.9,1.9) {$S^{\la_{L}^{2}}$};
	%	\draw (3.4,1) ..controls (4.4,2) and (5.9,2)..(6.9,1);
	%	\node at (5.1,2) {\scriptsize$T_{BD}$};
	\draw[thick,dashed]  (5.1,0) --(5.1,1.3);
%	\node at (6.9,3) {};
	\end{tikzpicture}
	\end{array}
	\end{equation*}
	\vskip -0.3cm
	\caption{Demonstration of the general case}
	\label{gene}
\end{figure}

\section{Grothendieck ring}
In this section,  we shall calculate
the structure constants of the Grothendieck ring. Throughout this section,
all walled Brauer algebras are assumed to be semisimple.

It is well known that in semisimple cases,  all cell modules of  algebra  $B_{r,s}(\delta)$
form a complete set of nonisomorphic simple modules.
Let $\Delta_{r,s}(\la^L,\la^R)=V_{r,s}^{l}\o (S^{\la^L}\boxtimes S^{\la^R})$
be  a  cell module of algebra $B_{r,s}$ with 	$\lambda^L\vdash r-l,  \lambda^R\vdash s-l$.

Suppose $r_1, r_2, s_1, s_2 \in \mathbb{N}$  such that $r_1 + r_2=r$ and $s_1 + s_2=s$.
Let    $\Delta_{r_1,s_1}(\nu_1^L,\nu_1^R)$ and $\Delta_{r_2,s_2}(\nu_2^L,\nu_2^R)$ be
 cell modules of $B_{r_1,s_1}(\delta)$ and $B_{r_2,s_2}(\delta)$, respectively, with
	$\nu_1^L\vdash r_1-l_1$, 	$\nu_1^R\vdash s_1-l_1$,  $\nu_2^L\vdash r_2-l_2$,
and $\nu_2^R\vdash s_2-l_2$. Then all tensor products  $\Delta_{r_1,s_1}(\nu_1^L,\nu_1^R)\,\otimes\, \Delta_{r_2,s_2}(\nu_2^L,\nu_2^R)$
form a complete set of nonisomorphism simple modules
of the semisimple algebra $B_{r_1,s_1}(\delta) \o_k B_{r_2,s_2}(\delta)$.
Denote by $ C^{\nu}_{\la|\mu}$  the
Littlewood-Richardson coefficients, that is,
the multiplicities of the tensor product $S^{\la}\,\boxtimes\, S^{\mu}$
in restricted module $\res (S^{\nu})$. Consider $\Delta_{r,s}(\la^L,\la^R)$
as a $(\scriptsize{B_{r_1,s_1}\o B_{r_2,s_2}})$-module.

\begin{thm}\label{cons}Let
$\mathrm{B}(\delta)$ be a semi-simple sequence of the walled Brauer algebras with $\de\in k$.
Then we have
\begin{eqnarray}
[\Delta_{r_1,s_1}(\nu_1^L,\nu_1^R)]\cdot [\Delta_{r_2,s_2}(\nu_2^L,\nu_2^R)]	&=& \sum_{\vec{\lambda}}C_{\vec{\nu}_1|\vec{\nu}_2}^{\vec{\lambda}}[\Delta_{r,s}(\la^L,\la^R)],\nonumber
\end{eqnarray}
where $\vec{\nu}_i=(\nu_i^L,\nu_i^R)$ for $i=1,2$, and the  pair $\vec{\lambda}=(\la^L,\la^R)$
in sum runs over all indices  of cell modules of $B_{r,s}(\delta)$ and the  structure constant
$$C_{\vec{\nu}_1|\vec{\nu}_2}^{\vec{\lambda}}=\sum_{\vec{r}}
\Big(\prod_{i=1,2} C^{\nu_i^L}_{\la_{L}^{i}\,|\,\mu_{i}}\,\,\cdot\,\,
C^{\nu_{ i+(-1)^{i+1}}^R}_{\la_{R}^{i}\,|\,\mu_{i}}\Big),$$
with the indices \[
\vec{r}=\left\{
\begin{array}{ll}
\la_{L}^{1}\vdash ( r_1-T_{AD}-T_{AC}), &
\mu_{1}\vdash T_{AC}, \\
\la_{R}^{2}\vdash ( s_1-T_{AD}-T_{BD}),&\\
\la_{L}^{2}\vdash ( r_2-T_{BC}-T_{BD}),&  \mu_{2}\vdash T_{BD},\\
\la_{2}^{R}\vdash ( s_2-T_{BC}-T_{AC}).&
\end{array}
\right.
\]
running over all  $(T_{AC},T_{BD},T_{AD},T_{BC})\in I_{(r_1|s_1,r_2|s_2)}^l$ such that
$$  l_1=T_{AD}
\mbox{~and~}l_2=T_{BC}.$$
\end{thm}

\begin{proof}
Firstly, we have
\begin{eqnarray}
	[\Delta_{r_1,s_1}(\nu_1^L,\nu_1^R)]\cdot [\Delta_{r_2,s_2}(\nu_2^L,\nu_2^R)]	
&=& [\ind\big( \Delta_{r_1,s_1}(\nu_1^L,\nu_1^R)\,\otimes\, \Delta_{r_2,s_2}(\nu_2^L,\nu_2^R)\big)]\nonumber\\
\nonumber	\\
& = &\sum_{\vec{\lambda}}C_{\vec{\nu}_1|\vec{\nu}_2}^{\vec{\lambda}}[\Delta_{r,s}(\la^L,\la^R)], \nonumber
\end{eqnarray}
where
\begin{eqnarray*}
C_{\vec{\nu}_1|\vec{\nu}_2}^{\vec{\lambda}}	&:=& \dim \Hom \Big(\ind\big( \Delta_{r_1,s_1}(\nu_1^L,\nu_1^R)\,
\otimes\, \Delta_{r_2,s_2}(\nu_2^L,\nu_2^R)\big)\, , \,   \Delta_{r,s}(\la^L,\la^R)\Big)\nonumber\\
& = &\dim \Hom\Big(  \Delta_{r_1,s_1}(\nu_1^L,\nu_1^R)\,\otimes\, \Delta_{r_2,s_2}(\nu_2^L,\nu_2^R)\, , \,
\res\big(\Delta_{r,s}(\la^L,\la^R)\big)\Big).
\end{eqnarray*}

From Theorem \ref{keyp} (2), the restricted module $\res\big(\Delta_{r,s}(\la^L,\la^R)\big)$
is filtered by cell modules of the following form
	$$	
\begin{array}{ll}
V_{r_1, s_1}^{T_{AD}}\o \Big (~ \ind(\,S^{\la_{L}^1}\o S^{\mu_{1}}\,) \,
\boxtimes \, \ind(\,S^{\la_{R}^2}\o S^{\mu_{2}}\,)~\Big ) & \text{(left)} 		 \vspace{0.1cm} \\	\hspace{5cm}
\bigotimes\\
\vspace{0.8cm} V_{r_2, s_2}^{T_{BC}}\o \Big (~\ind(\,S^{\la_{L}^{2}}\o S^{\mu_{2}}\,) \,
\boxtimes \, \ind(\,S^{\la_{R}^{1}}\o S^{\mu_{1}}\,) ~\Big )& \text{(right)}
\end{array},$$
with $(T_{AC},T_{BD},T_{AD},T_{BC})\in I_{(r_1|s_1,r_2|s_2)}^l$ and the indices \[
\left\{
\begin{array}{ll}
\la_{L}^{1}\vdash ( r_1-T_{AD}-T_{AC}), &
\mu_{1}\vdash T_{AC}, \\
\la_{R}^{2}\vdash ( s_1-T_{AD}-T_{BD}),&\\
\la_{L}^{2}\vdash ( r_2-T_{BC}-T_{BD}),&  \mu_{2}\vdash T_{BD},\\
\la_{2}^{R}\vdash ( s_2-T_{BC}-T_{AC}).&
\end{array}
\right.
\]
Note that for $i=1,2$, $\Delta_{r_i,s_i}(\nu_i^L,\nu_i^R)=V_{r_i,s_i}^{l_i}\o (S^{\nu_i^L}\,
\boxtimes\, S^{\nu_i^R})$, and simply write $T_{1}=T_{AD}, T_{2}=T_{BC}$.
Thanks to Proposition \ref{aalc}, we get
$$ \Hom\bigg( V_{r_i,s_i}^{l_i}\o (S^{\nu_i^L}\,
\boxtimes\, S^{\nu_i^R})\, , \, V_{r_i, s_i}^{T_{i}}\o \Big (~ \ind(\,S^{\la_{L}^{i}}\boxtimes S^{\mu_{i}}\,) \, \boxtimes \, \ind(\,S^{\la_{R}^{i+(-1)^{i+1}}}\boxtimes S^{\mu_{i+(-1)^{i+1}}}\,)~\Big) \bigg)$$
\[
=\left\{
\begin{array}{ll}
\Hom\big(  S^{\nu_i^L}\,\boxtimes\, S^{\nu_i^R} \, , \, \ind(\,S^{\la_{L}^{i}}
\boxtimes S^{\mu_{i}}\,) \, \boxtimes \, \ind(\,S^{\la_{R}^{i+(-1)^{i+1}}}
\boxtimes S^{\mu_{i+(-1)^{i+1}}}\, \big),\, (\blacklozenge) &\\
\hspace{23em} \mbox{if~$ l_i=T_{i}$, for $i=1,2;\,$}\\
0,\hspace{22em}\mbox{otherwise.}&
\end{array}
\right.
\]
By Frobenius reciprocity, $(\blacklozenge)$ is equal to
$$\Hom\big(\, \res (S^{\nu_i^L}) ,  S^{\la_{L}^{i}}\boxtimes S^{\mu_{i}}\big)\o
\Hom\big(\, \res (S^{\nu_i^R} \,),    S^{\la_{R}^{i+(-1)^{i+1}}}\boxtimes S^{\mu_{i+(-1)^{i+1}}}\,\big)$$
and the dimension of the factors is the  Littlewood-Richardson coefficient:
$$C^{\nu_1^L}_{\la_{L}^{1}\,|\,\mu_{1}}\,\,,\,\, C^{\nu_1^R}_{\la_{R}^{2}\,|\,\mu_{2}}\,
\mbox{ for}~ i=1; \, C^{\nu_2^L}_{\la_{L}^{2}\,|\,\mu_{2}}\,\,,\,\, C^{\nu_2^R}_{\la_{R}^{1}\,|\,\mu_{1}}\, \mbox{ for}~i=2.$$
Therefore, we get the dimension of previous Hom-space as follows:
\begin{eqnarray*}
	\left\{
	\begin{array}{ll}
		\prod\limits_{i=1,2} C^{\nu_i^L}_{\la_{L}^{i}\,|\,\mu_{i}}\,\,\cdot\,\,
C^{\nu_{ i+(-1)^{i+1}}^R}_{\la_{R}^{i}\,|\,\mu_{i}},\, & \mbox{if~} l_1=T_{AD}
		\mbox{~and~}l_2=T_{BC},\\
		0 &\mbox{otherwise.}
	\end{array}
	\right.
\end{eqnarray*}

Consequently, the conclusion follows when index $(T_{AC},T_{BD},T_{AD},T_{BC})$
runs over all  $ I_{(r_1|s_1,r_2|s_2)}^l$ such that
$  l_1=T_{AD}
\mbox{~and~}l_2=T_{BC}.$
\end{proof}

\end{document}